\documentclass[11pt]{article}

\usepackage{amsthm,amsmath,amssymb,amsfonts,enumerate,latexsym,amsbsy,graphicx}
\usepackage[usenames,dvipsnames,svgnames,table]{xcolor}

\usepackage[margin=1.3in]{geometry}

\usepackage{hyperref}
\hypersetup{
    colorlinks,%
    citecolor=black,%
    filecolor=black,%
    linkcolor=black,%
    urlcolor=black
}

\newtheorem{THEOREM}{Theorem}[section]

\newtheorem{theorem}[THEOREM]{Theorem}
\newtheorem{corollary}[THEOREM]{Corollary}
\newtheorem{lemma}[THEOREM]{Lemma}

\newtheorem{remark}[THEOREM]{Remark}

\theoremstyle{definition}
\newtheorem{definition}[THEOREM]{Definition}

\theoremstyle{remark}

\def \va {\varepsilon}

\newcommand{\R}{\ensuremath{\mathbb{R}}}   %%% reals
\DeclareMathOperator{\osc}{osc}

\DeclareMathOperator{\diver}{div} %
\DeclareMathOperator{\tr}{Tr} %
\def\eps{\varepsilon}

\begin{document}
\title{H\"older gradient estimates for a class of singular or degenerate parabolic equations}

\author{\medskip Cyril Imbert, ~ ~Tianling Jin\footnote{Support in part by Hong Kong RGC grant ECS 26300716.} ~~ and ~~ Luis Silvestre\footnote{Support in part by NSF grants DMS-1254332 and DMS-1362525.}}
\date{\today}

\maketitle

\begin{abstract}
  We prove interior H\"older estimate for the spatial gradients of the
  viscosity solutions to the singular or degenerate parabolic equation
\[
u_t=|\nabla u|^{\kappa}\diver (|\nabla u|^{p-2}\nabla u),
\]
where $p\in (1,\infty)$ and $\kappa\in (1-p,\infty).$ This includes
the from $L^\infty$ to $C^{1,\alpha}$ regularity for parabolic
$p$-Laplacian equations in both divergence form with $\kappa=0$, and
non-divergence form with $\kappa=2-p$. This work is a continuation of
a paper by the last two authors \cite{JS}.
\end{abstract}

\section{Introduction}

Let $1<p<\infty$ and $\kappa\in (1-p,\infty)$. We are interested in
the regularity of solutions of
\begin{equation}\label{eq:degenerate}
u_t=|\nabla u|^{\kappa}\diver (|\nabla u|^{p-2}\nabla u).
\end{equation}

When $\kappa=0$, this is the classical parabolic $p$-Laplacian
equation in divergence form. This is the natural case in the context of gradient flows of Sobolev norms. H\"older estimates for the spatial
gradient of their weak solutions (in the sense of distribution) were obtained by DiBenedetto and Friedman in \cite{DBF} (see also Wiegner \cite{Wiegner}).

When $\kappa=2-p$, the equation \eqref{eq:degenerate} is a parabolic
homogeneous $p$-Laplacian equations. This is the most relevant case for applications to tug-of-war-like
stochastic games with white noise, see Peres-Sheffield \cite{PS}. This equation has been studied by Garofalo \cite{Garofalo}, Banerjee-Garofalo \cite{BG, BG2,BG3}, Does \cite{KD}, Manfredi-Parviainen-Rossi \cite{MPR, MPR2012}, Rossi \cite{Rossi2011}, Juutinen \cite{Juutinen2014}, Kawohl-Kr\"omer-Kurtz \cite{KKK}, Liu-Schikorra \cite{LS2015}, Rudd \cite{Rudd2015}, as well as the last two authors \cite{JS}. H\"older estimates for the spatial gradient of their solutions was proved in \cite{JS}. The solution of this equation is understood in the viscosity sense. The toolbox of methods that one can apply are completely different to the variational techniques used classically for  $p$-Laplacian problems.

The equation \eqref{eq:degenerate} can be rewritten as
\begin{equation}\label{eq:degenerate1}
u_t=|\nabla u|^{\gamma}\left(\Delta u + (p-2)|\nabla u|^{-2}u_iu_ju_{ij}\right),
\end{equation}
where $\gamma=p+\kappa-2 > -1$.  In this paper, we prove
H\"older estimates for the spatial gradients of viscosity solutions to
\eqref{eq:degenerate1} for $1<p<\infty$ and $\gamma\in
(-1,\infty)$.
Therefore, it provides a unified approach for all those $\gamma$ and $p$, including the two
special cases $\gamma=0$ and $\gamma=p-2$ mentioned above.

The viscosity solutions to \eqref{eq:degenerate1} with $\gamma>-1$ and
$p>1$ falls into the general framework studied by Ohnuma-Sato in
\cite{OS}, which is an extension of the work of Barles-Georgelin \cite{BGe} and Ishii-Souganidis
\cite{IS} on the viscosity solutions of singular/degenerate parabolic
equations. We postpone the definition of viscosity solutions of
\eqref{eq:degenerate1} to Section \ref{sec:approximation}.  For $r>0$,
$Q_r$ denotes $B_r\times(-r^2,0]$, where $B_r\subset\R^n$ is the ball
of radius $r$ centered at the origin.

\begin{theorem}\label{thm:mainholdergradient}
  Let $u$ be a viscosity solution of \eqref{eq:degenerate1} in $Q_1$,
  where $1<p<\infty$ and $\gamma\in (-1,\infty)$. Then there exist two
  constants $\alpha\in (0,1)$ and $C>0$, both of which depends only on
  $n,\gamma, p$ and $\|u\|_{L^\infty(Q_1)}$, such that
\[
 \|\nabla u\|_{C^\alpha(Q_{1/2})}\le C.
\]
Also, the following H\"older regularity in time holds
\[ \sup_{(x,t), (x,s)\in Q_{1/2}}\frac{|u(x,t)-u(x,s)|}{|t-s|^{\frac{1+\alpha}{2-\alpha\gamma}}} \le C. \]
Note that $(1+\alpha)/(2-\alpha \gamma) > 1/2$ for every $\alpha > 0$ and $\gamma > -1$.
\end{theorem}

Our proof in this paper follows a similar structure as in \cite{JS}, with some notable differences that we explain below. We use non-divergence techniques in the context of viscosity solutions. Theorem \ref{thm:mainholdergradient} tells us that these techniques are in some sense stronger than variational methods when dealing with the regularity of scalar $p$-Laplacian type equations. The weakness of these methods (at least as of now) is that they are ineffective for systems.

The greatest difficulty extending the result in \cite{JS} to Theorem \ref{thm:mainholdergradient} comes from the lack of uniform ellipticity. When $\gamma = 0$, the equation \eqref{eq:degenerate1} is a parabolic equation in non-divergence form with uniformly elliptic coefficients (depending on the solution $u$). Because of this, in \cite{JS}, we use the theory developed by Krylov and Safonov, and other classical results, to get some basic uniform a priori estimates. This fact is no longer true for other values of $\gamma$. The first step in our proof is to obtain a Lipschitz modulus of continuity. That step uses the uniform ellipticity very strongly in \cite{JS}. In this paper we take a different approach using the method of Ishii and Lions \cite{IL}. Another step where the uniform ellipticity plays a strong role is in a lemma which transfers an oscillation bound in space, for every fixed time, to a space-time oscillation. In this paper that is achieved through Lemmas \ref{l:timeslices2allcylinder} and \ref{l:timeslices2allcylinder-aroundplane}, which are considerably more difficult than their counterpart in \cite{JS}. Other, more minor, difficulties include the fact that the non-homogeneous right hand side forces us to work with a different scaling (See the definition of $Q_r^\rho$ by the beginning of Section \ref{sec:holder gradient}).

In order to avoid some of the technical difficulties caused by the non-differentiability of viscosity solutions,  we first consider the regularized problem \eqref{eq:main va} in the below, and then obtain uniform estimates so that we can pass to the limit in the end. For $\va\in(0,1)$, let $u$ be smooth and satisfy that
\begin{equation}\label{eq:main va} 
\partial_t u = (|\nabla u |^2+\va^2)^{\gamma/2} \left(\delta_{ij}+(p-2)\frac{u_iu_j}{|\nabla u|^2+\va^2}\right)u_{ij}.
\end{equation}
We are going to establish Lipschitz estimate and H\"older gradient
estimates for $u$, which will be independent of $\va\in (0,1)$, in
Sections \ref{sec:lip}, \ref{sec:holder in t}, \ref{sec:holder
  gradient}. Then in Section \ref{sec:approximation}, we recall the
definition of viscosity solutions to \eqref{eq:degenerate1}, as well
as their several useful properties, and prove Theorem
\ref{thm:mainholdergradient} via approximation arguments. This idea of approximating the problem with a smoother one and proving uniform estimates is very standard.

\bigskip
\noindent\textbf{Acknowledgement:} Part of this work was done when T. Jin was visiting California Institute of Technology as an Orr foundation Caltech-HKUST Visiting Scholar. He would like to thank Professor Thomas Y. Hou for the kind hosting and discussions.

\section{Lipschitz estimates in the spatial variables}\label{sec:lip}

The proof of Lipschitz estimate in \cite{JS} for $\gamma=0$ is based
on a calculation that $|\nabla u|^p$ is a subsolution of a uniformly
parabolic equation. We are not able to find a similar
quantity for other nonzero $\gamma$. The proof we give here is completely different. It makes use of the
Ishii-Lions' method \cite{IL}. However, we need to apply this method
twice: first we obtain log-Lipschitz estimates, and then use this
log-Lipschitz estimate and Ishii-Lions' method again to prove
Lipschitz estimate. Moreover, the Lipschitz estimate holds for
$\gamma>-2$ instead of $\gamma>-1$.

%---------------------------------
\begin{lemma}[Log-Lipschitz estimate]\label{lem:log-lip}
  Let $u$ be a smooth solution of \eqref{eq:main va} in $Q_4$ with $\gamma>-2$ and $\va\in (0,1)$.
  Then there exist two positive constants $L_1$ and $L_2$ depending
  only on $n, p, \gamma$ and $\|u\|_{L^\infty(Q_4)}$ such that for every
  $(t_0,x_0) \in Q_1$, we have
\[
u(t,x) - u(t,y) \le L_1 |x-y| |\log|x-y||+ \frac{L_2}2 |x-x_0|^2+ \frac{L_2}2 |y-x_0|^2 + \frac{L_2}2
(t-t_0)^2 
\]
for all $ t \in [t_0-1,t_0]$ and $x,y \in B_1 (x_0).$
\end{lemma}
%------------------------
\begin{proof} Without loss of generality, we assume $x_0=0$ and
  $t_0=0$. It is sufficient to prove that
\[ 
M := \max_{-1\le t\le 0,\  x,\,y \in \overline{B_1}} \big\{ u(t,x) -u(t,y) 
- L_1 \phi (|x-y|) -  \frac{L_2}2 |x|^2-  \frac{L_2}2 |y|^2 - \frac{L_2}2 t^2 \big\}
\]
is non-positive, where
\[ \phi (r) = \begin{cases} 
-r \log r & \text{ for } r \in [0,e^{-1}] \\
e^{-1} & \text{ for } r \ge e^{-1}.
\end{cases} \]   

We assume this is not true and we will exhibit a contradiction.  In
the rest of the proof, $t \in [-1,0]$ and $x,y
\in \overline{B}_1$ denote the points realizing the maximum
defining $M$.

Since $M \geq 0$, we have 
\[ L_1\phi (|x-y|) + \frac{L_2}2 (|x|^2 +|y|^2+ t^2) \le 2 \|u\|_{L^\infty(Q_4)}. \]
In particular, 
\begin{equation}\label{estim:penal1}
  \phi (\delta) \le \frac{2 \|u\|_{L^\infty(Q_4)}}{L_1},\quad\mbox{where }\delta = |a|  \quad  \text{ and }  \quad a = x-y,
\end{equation}
 and 
\begin{equation}\label{estim:penal2}
 |t|+|x|+|y| \le 6  \sqrt{\frac{\|u\|_{L^\infty(Q_4)}}{L_2}} .
\end{equation}
Hence, for $L_2$ large enough, depending only on
$\|u\|_{L^\infty(Q_4)}$, we can ensure that $t \in (-1,0]$
and $x,y \in B_1$.  We choose $L^2$ here and fix it for the rest of the proof. Thus, from now on $L_2$ is a constant depending only on $\|u\|_{L^\infty}$.

Choosing $L_1$ large, we can ensure that $\delta(<e^{-2})$ is small enough to satisfy
\[ \phi(\delta) \ge 2\delta. \] 
In this case, \eqref{estim:penal1} implies
\begin{equation}\label{estim:penal1-bis}
 \delta \le \frac{ \|u\|_{L^\infty(Q_4)}}{L_1}.
\end{equation}

Since $t \in [-1,0]$ and $x,y
\in B_1$ realizing the supremum defining $M$, we have that
\begin{align}
\nabla u(t,x)&=L_1\phi'(\delta)\hat a+L_2x\nonumber\\
\nabla u(t,y)&=L_1\phi'(\delta)\hat a-L_2y\nonumber\\
u_t(t,x)-u_t(t,y)&= L_2t \nonumber\\
\begin{bmatrix}
\nabla^2u(t,x) & 0 \\
0 & -\nabla^2u(t,y) \end{bmatrix} 
& \le L_1 
\begin{bmatrix} Z& -Z \\-Z&Z \end{bmatrix} + L_2 I\label{mi},
\end{align}
where 
\[ 
Z = \phi''(\delta) \hat a \otimes \hat a + \frac{\phi'(\delta)}{\delta} (I - \hat a \otimes \hat a)
\quad\mbox{and}\quad \hat a=\frac{a}{|a|}=\frac{x-y}{|x-y|}. 
\]
For $z\in\R^n$, we let 
\[ 
A (z) = I + (p-2)\frac{z_iz_j}{|z|^2+\va^2},
\]
and $q=L_1\phi'(\delta)\hat a, X=\nabla^2u(t,x)$ and
$Y=\nabla^2u(t,y)$. By evaluating the equation at $(t,x)$ and $(t,y)$,
we have
\begin{equation}\label{eq:evaluate the equation}
  L_2t\le (|q+L_2x|^2+\va^2)^{\frac{\gamma}{2}}\tr\Big(A(q+L_2x)X\Big)
-(|q-L_2y|^2+\va^2)^{\frac{\gamma}{2}}\tr\Big(A(q-L_2y)Y\Big).
\end{equation}
Whenever we write $C$ in this proof, we denote a positive constant, large enough depending only on $n,p,\gamma$ and $\|u\|_{L^\infty(Q_4)}$, which may vary from lines to lines.  Recall that we have already chosen $L_2$ above depending on $\|u\|_{L^\infty}$ only.

Note that $|q| = L_1 |\phi'(\delta)|$. Choosing $L_1$ large enough, $\delta$ will be small, $|\phi'(\delta)|$ will thus be large, and $|q| \gg L_2$. In particular,
\begin{equation}\label{eq:gradient big}
|q|/2\le |q+L_2 x| \le 2 |q|\quad \mbox{and} \quad |q|/2\le |q-L_2 y| \le 2 |q|.
\end{equation}

From \eqref{mi} and the fact that $\phi''(\delta)<0$, we have
\begin{equation}\label{eq:hession one bound}
  \begin{split}
    X=\nabla^2u(t,x)&\le L_1 \frac{\phi'(\delta)}{\delta} (I - \hat a \otimes \hat a)+L_2I\\
    -Y=-\nabla^2u(t,y)&\le L_1 \frac{\phi'(\delta)}{\delta} (I - \hat
    a \otimes \hat a)+L_2I
  \end{split}
\end{equation}
Making use of \eqref{eq:evaluate the equation}, \eqref{eq:gradient
  big} and \eqref{eq:hession one bound}, we have
\[
\begin{split}
  \tr\Big(A(q+L_2x)X\Big)&= (|q+L_2x|^2+\va^2)^{-\frac{\gamma}{2}}L_2t
  +\left(\frac{|q-L_2y|^2+\va^2}{|q+L_2x|^2+\va^2}\right)^{\frac{\gamma}{2}}\tr\Big(A(q-L_2y)Y\Big)\\
  &\ge  -C \left( |q|^{-\gamma} + L_1\frac{\phi'(\delta)}{\delta} + 1 \right).
\end{split}
\]
Therefore, it follows from \eqref{eq:hession one bound} and the
ellipticity of $A$ that
\begin{equation}\label{eq:X-norm}
|X|\le C  \left( |q|^{-\gamma}+L_1\frac{\phi'(\delta)}{\delta}+1 \right)
\end{equation}
Similarly,
\[
|Y|\le C  \left( |q|^{-\gamma}+L_1\frac{\phi'(\delta)}{\delta}+1 \right)
\]
Let
\[ B(z) = \left(|z|^2+\eps \right)^\gamma A(z).\]
We get from \eqref{eq:evaluate the equation} and \eqref{estim:penal2}
the following inequality
\begin{equation}\label{inter1}
-C \le  \tr [ B(q+L_2 x) X ] -  \tr [ B (q-L_2 y) Y ] \le T_1 + T_2 
\end{equation}
where 
\[ T_1=  \tr [B(q-L_2 y) (X-Y) ] \quad \text{ and } \quad T_2 =  |X||B (q+L_2 x) - B(q-L_2 y))|.\]

We first estimate $T_2$. Using successively \eqref{estim:penal2},
\eqref{eq:gradient big}, \eqref{eq:X-norm} and mean value theorem, we
get
\begin{align}
 T_2 &\le C |X| |q|^{\gamma-1}  |x+y| \nonumber\\
& \le C |X| |q|^{\gamma-1} \label{eq:T2} \nonumber \\
\nonumber
&\le  C \left(|q|^{-\gamma}+\frac{L_1\phi'(\delta)}\delta + 1 \right)|q|^{\gamma-1}, \\
&\le C  \left(|q|^{-1}+\frac{|q|^\gamma}\delta + |q|^{\gamma-1} \right). 
\end{align}

We now turn to $T_1$. On one hand, evaluating \eqref{mi} with respect to a vector of the form $(\xi,\xi)$, we get that for all $\xi \in \R^d$ we have
\begin{equation}\label{mi-plus}
    (X-Y) \xi \cdot \xi  \le  2 L_2 |\xi|^2.
\end{equation}
On the other hand, when we evaluate \eqref{mi} with respect to $(\hat a, \hat a)$, we get,
\begin{equation}\label{mi-plus2}
    (X-Y) \hat a \cdot \hat a  \le  4 L_1 \phi''(\delta) + 2 L_2
\end{equation}

The inequality \eqref{mi-plus} tells us that all eigenvalues of $ (X-Y)$ are bounded above by a constant $C$. The inequality \eqref{mi-plus2} tells us that there is at least one eigenvalue that is less than the negative number $4 L_1 \phi''(\delta) + 2 L_2$. Because of the uniform ellipticity of $A$, we obtain
\[
T_1 \leq C |q|^\gamma \left( L_1 \phi''(\delta) + 1 \right).
\]

In view of the estimates for $T_1$ and $T_2$, we finally get from
\eqref{inter1} that

\begin{align}
   -L_1 \phi''(\delta)|q|^\gamma &\le C \left( |q|^\gamma +|q|^{-1}+\frac{|q|^\gamma}\delta
      + |q|^{\gamma-1} + 1 \right), \nonumber \\
\intertext{or equivalently,}
   -L_1 \phi''(\delta) &\le C \left( 1+|q|^{-1-\gamma}+\frac{1}\delta
      + |q|^{-1} + |q|^{-\gamma} \right).\label{eq:aux-contra}
\end{align}
Our purpose is to choose $L_1$ large in order to get a contradiction in \eqref{eq:aux-contra}.

Recall that we have the estimate $\delta \leq C / L_1$. From our choice of $\phi$, $\phi'(\delta) > 1$ for $\delta$ small and $-\phi''(\delta) = 1/\delta \geq c L_1$.

For $L_1$ sufficiently large, since $\gamma > -2$
\begin{align*} 
 C(1+|q|^{-1-\gamma}  + |q|^{-1} + |q|^{-\gamma}) &\leq C \left( 1+ L_1^{-1-\gamma} + L_1^{-1} + L_1^{-\gamma} \right) \leq \frac c 2 L_1^2 , \\ 
 &\leq -\frac 12 L_1 \phi''(\delta).
\end{align*}

The remaining term is handled because of the special form of the function $\phi$. We have
\[ -L_1 \phi''(\delta) = \frac{ L_1} \delta > \frac {2C} \delta,\]
for $L_1$ sufficiently large.

Therefore, we reached a contradiction. The proof of this lemma is thereby completed.
\end{proof}

%------------------------------------
By letting $t=t_0$ and $y=x_0$ in Lemma \ref{lem:log-lip}, and since $(x_0,t_0)$ is arbitrary, we have
\begin{corollary}\label{cor:log-lip}
  Let $u$ be a smooth solution of \eqref{eq:main va} in $Q_4$ with
  $\gamma>-2$ and $\va\in (0,1)$. Then there exists a positive
  constant $C$ depending only on $n, \gamma, p$ and
  $\|u\|_{L^\infty(Q_4)}$ such that for every $(t,x), (t,y)\in Q_3$
  and $|x-y|<1/2$, we have
\[
|u(t,x)-u(t,y)|\le C|x-y||\log|x-y||.
\]
\end{corollary}
%------------------------------------

We shall make use of the above log-Lipschitz estimate and the
Ishii-Lions' method \cite{IL} again to prove the following Lipschitz
estimate.

%\textcolor{blue}{why $Q_4$, $Q_1$ and $B_{1/4}$? I understand it is 
%necessary to reduce. When you are not an expert in regularity theory,
%those harmless choices are always very mysterious}
\begin{lemma}[Lipschitz estimate]\label{lem:lip}
Let $u$ be a smooth solution of \eqref{eq:main va} in $Q_4$ with $\gamma>-2$ and $\va\in (0,1)$.
  Then there exist two positive constants $L_1$ and $L_2$ depending
  only on $n, p, \gamma$ and $\|u\|_{L^\infty(Q_4)}$ such that for every
  $(t_0,x_0) \in Q_1$, we have
\[
u(t,x) - u(t,y) \le L_1 |x-y| + \frac{L_2}2 |x-x_0|^2+ \frac{L_2}2 |y-x_0|^2 + \frac{L_2}2
(t-t_0)^2 
\]
for all $ t \in [t_0-1,t_0]$ and $x,y \in B_{1/4} (x_0).$
\end{lemma}
\begin{proof} 
  The proof of this lemma follows the same computations as that of Lemma  \ref{lem:log-lip}, but we make use of the conclusion of
  Corollary \ref{cor:log-lip} in order to improve our estimate.

Without loss of generality, we assume $x_0=0$ and
  $t_0=0$. As before, we define
\[ 
M := \max_{-1\le t\le 0,\  x,\,y \in B_1} \big\{ u(t,x) -u(t,y) 
- L_1 \phi (|x-y|) -  \frac{L_2}2 |x|^2-  \frac{L_2}2 |y|^2 - \frac{L_2}2 t^2 \big\}
\]
is non-positive, where
\[ \phi (r) = \begin{cases} 
r - \frac1{2-\gamma_0} r^{2-\gamma_0} & \text{ for } r \in [0,1] \\
1-\frac1{2-\gamma_0} & \text{ for } r \ge 1
\end{cases} \]   
for some $\gamma_0 \in (1/2,1)$.

We assume this is not true in order to obtain a contradiction.  In
the remaining of the proof of the lemma, $t \in [-1,0]$ and $x,y
\in \overline{B}_{1/4}$ denote the points realizing the maximum
defining $M$.

For the same reasons as in the proof of Lemma \ref{lem:log-lip}, the inequalities \eqref{estim:penal1} and \eqref{estim:penal2}  also apply in this case. Thus, we can use the same choice of $L_2$, depending on $\|u\|_{L^\infty}$ only, that ensures  $t \in (-1,0]$ and $x,y \in B_1$.

From Corollary \ref{cor:log-lip}, we already know that $u(t,x) - u(t,y) \leq C |x-y| |\log |x-y||$. Since $M \geq 0$, 
\begin{equation}\label{eq:use continuity}
L_1\phi (|x-y|) + \frac{L_2}2 (|x|^2 +|y|^2+ t^2) \le C|x-y||\log|x-y||. 
\end{equation}
In particular, we obtain an improvement of  \eqref{estim:penal2},
\begin{equation}\label{estim:penal2-2}
 |t|+|x|+|y| \le C \sqrt{\frac{\delta|\log\delta|}{L_2}} .
\end{equation}
This gives us an upper bound for $|x+y|$ that we can use to improve \eqref{eq:T2}. 
\begin{align*} 
 T_2  &\le C |X| |q|^{\gamma-1} |x+y|, \\
&\le C \left( |q|^{-1} + \frac{|q|^\gamma} \delta + |q|^{\gamma-1} \right) \sqrt{\delta |\log \delta|},
\end{align*}
The estimate for $T_1$ stays unchanged. Hence, \eqref{eq:aux-contra} becomes
\[
\begin{split}
-L_1 \phi''(\delta) &\le C \left( 1 + \sqrt{\delta|\log\delta|} \left(
|q|^{-1} +|q|^{-1-\gamma} +\frac{1} \delta +  |q|^{-\gamma} \right) \right).
\end{split}
\]
Recall that $|q|=L_1\phi'(\delta)\geq L_1 / 2$ and $\phi''(\delta)=(\gamma_0-1)\delta^{-\gamma_0}$. Then,
\[ 
L_1 \delta^{-\gamma_0} \leq C \left( 1 + \sqrt{\delta |\log \delta|} \left( 1+ L_1^{-1} + L_1^{-1-\gamma} + \delta^{-1} + L_1^{-\gamma} \right) \right) 
\]
The term $+1$ inside the innermost parenthesis is there just to ensure that the inequality holds both for $\gamma < 0$ and $\gamma>0$. Recalling that $\delta < C / L_1$, we obtain an inequality in terms of $L_1$ only.
\[ L_1^{1+\gamma_0} \leq C \left( 1 + L_1^{-1/2} \sqrt{\log L_1} \left( 1+ L_1^{-1} + L_1^{-1-\gamma} + L_1 + L_1^{-\gamma} \right) \right) 
\]

Choosing $L_1$ large, we arrive to a contradiction given that $1+\gamma_0 > \max( 1/2, -1/2-\gamma )$ since $\gamma_0 > 1/2$ and $\gamma > -2$.
\end{proof}

Again, by letting $t=t_0$ and $y=x_0$ in Lemma \ref{lem:lip}, and since $(x_0,t_0)$ is arbitrary, we have

\begin{corollary}\label{cor:lip}
  Let $u$ be a smooth solution of \eqref{eq:main va} in $Q_4$ with
  $\gamma>-2$ and $\va\in (0,1)$. Then there exists a positive
  constant $C$ depending only on $n, \gamma, p$ and
  $\|u\|_{L^\infty(Q_4)}$ such that for every $(t,x), (t,y)\in Q_3$
  and $|x-y|<1$, we have
\[
|u(t,x)-u(t,y)|\le C|x-y|.
\]
\end{corollary}

\section{H\"older estimates in the time variable}\label{sec:holder in t}

Using the Lipschitz continuity in $x$ and a simple comparison argument, we show that the solution of \eqref{eq:main va} is H\"older continuous in $t$.

%-----------------
%[\textcolor{red}{The proof works only for $\gamma > -1$}]

\begin{lemma}\label{lem:holder in t}
Let $u$ be a smooth solution of \eqref{eq:main va} in $Q_4$ with $\gamma>-1$ and $\va\in (0,1)$. 
Then there holds
\[ 
\sup_{t \neq s, (t,x), (s,x)\in Q_1}
\frac{|u(t,x)-u(s,x)|}{|t-s|^{1/2}} \le
C,
\]
where $C$ is a positive constant depending only on $n$, $p$, $\gamma$ and $\|u\|_{L^\infty(Q_4)}$.
\end{lemma}
%--------------
\begin{remark}
Deriving estimates in the time variable for estimates
  in the space variable by maximum principle techniques is
  classical. As far as viscosity solutions are concerned, the reader
  is referred to \cite[Lemma~9.1,~p.~317]{bbl} for instance.
\end{remark}

\begin{proof}
Let $\beta = \max(2, (2+\gamma)/(1+\gamma))$.  We claim that for all $t_0\in[-1,0)$, $\eta >0$, there exists
  $L_1>0$ and $L_2>0$ such that
\begin{equation}\label{eq:t-holder}
u(t,x) - u(t_0,0) \le \eta + L_1(t-t_0) + L_2 |x|^\beta =: \varphi(t,x) \quad\mbox{for all }(t,x)\in [t_0,0]\times\overline B_1.
\end{equation}
We first choose $L_2\ge 2\|u\|_{L^\infty(Q_3)}$ such that
\eqref{eq:t-holder} holds true for $x \in \partial B_1$. We will next
choose $L_2$ such that \eqref{eq:t-holder} holds true for $t=t_0$. In
this step we shall use Corollary \ref{cor:lip} that $u$ is Lipschitz
continuous with respect to the spatial variables. From Corollary \ref{cor:lip}, $\|\nabla u\|_{L^\infty(Q_3)}$ is bounded depending on $\|u\|_{L^\infty(Q_4)}$ only.
It is enough to choose
\[ \| \nabla u\|_{L^\infty(Q_3)} |x| \le \eta + L_2 |x|^\beta \]
which holds true if 
\[ L_2 \ge \frac{\|\nabla u\|_{L^\infty(Q_3)}^\beta}{\eta^{\beta-1}}. \] 
We finally choose $L_1$ such that the function $\varphi(t,x)$ is a supersolution of an equation that $u$ is a solution. The inequality \eqref{eq:t-holder} thus follows from the comparison principle. We use a slightly different equation depending on whether $\gamma \leq 0$ or $\gamma > 0$.

Let us start with the case $\gamma \leq 0$. In this case we will prove that $\varphi$ is a supersolution of the nonlinear equation \eqref{eq:main va}. That is
\begin{equation} \label{e:phisuper1}
 \varphi_t - (\eps^2 + |\nabla \varphi|^2)^{\gamma/2} \left( \delta_{ij} + (p-2) \frac{\varphi_i \varphi_j}{\eps^2 + |\nabla \varphi|^2} \right) \varphi_{ij} > 0.
\end{equation}
In order to ensure this inequality, we choose $L_1$ so that
\[ L_1 > (p-1) |\nabla \varphi|^\gamma |D^2 \varphi| \geq (\eps^2 + |\nabla \varphi|^2)^{\gamma/2} \left( \delta_{ij} + (p-2) \frac{\varphi_i \varphi_j}{\eps^2 + |\nabla \varphi|^2} \right) \varphi_{ij}.\]
We chose the exponent $\beta$ so that when $\gamma \leq 0$, $|\nabla \varphi|^\gamma |D^2 \varphi| = C L_1^{1+\gamma}$ for some constant $C$ depending on $n$ and $\gamma$. Thus, we must choose $L_1 = C L_2^{1+\gamma}$ in order to ensure \eqref{e:phisuper1}.

Therefore, still for the case $\gamma \leq 0$, $\beta = (2+\gamma)/(1+\gamma)$, and for any choice of $\eta>0$, using the comparison principle,
\[ 
\begin{split}
u(t,0)-u(t_0,0) 
&\le \eta + C\Big(\eta^{(1-\beta)} \|\nabla u\|_{L^\infty(Q_3)}^{\beta} +2\|u\|_{L^\infty(Q_3)}+\va\Big)^{\gamma+1}(t-t_0)\\
&\le \eta + C\eta^{-1} \|\nabla u\|_{L^\infty(Q_3)}^{\gamma+2}|t-t_0|+C(\|u\|_{L^\infty(Q_3)}+\va)^{\gamma+1}|t-t_0|
\end{split}
\]
By choosing $\eta=\|\nabla u\|_{L^\infty(Q_3)}^{\gamma/2+1} |t-t_0|^{1/2}$, it follows that for $t\in (t_0,0]$,
\[
  u(t,0)-u(t_0,0) \le C \left( \|\nabla u\|_{L^\infty(Q_3)}\right)^{\frac{\gamma+2}2}|t-t_0|^{1/2} +C\left( \|u\|_{L^\infty(Q_3)} + \eps \right) ^{\gamma+1}|t-t_0|.
\] 
The lemma is then concluded in the case $\gamma \leq 0$.

Let us now analyze the case $\gamma > 0$. In this case, we prove that $\varphi$ is a supersolution to a linear parabolic equation whose coefficients depend on $u$. That is
\begin{equation} \label{e:phisuper2}
 \varphi_t - (\eps^2 + |\nabla u|^2)^{\gamma/2} \left( \delta_{ij} + (p-2) \frac{u_i u_j}{\eps^2 + |\nabla u|^2} \right) \varphi_{ij} > 0.
\end{equation}

Since $\gamma > 0$ and $\nabla u$ is known to be bounded after Corollary \ref{cor:lip}, we can rewrite the equation assumption
\begin{equation} \label{e:phisuper2p}
 \varphi_t - a_{ij}(t,x) \varphi_{ij} > 0,
\end{equation}
where the coefficients $a_{ij}(t,x)$ are bounded by
\[ |a_{ij}(t,x)| \leq C \left( \eps +  \|\nabla u\|_{L^\infty(Q_3)} \right)^{\gamma}.\]
Since $\gamma > 0$, we pick $\beta = 2$ and $D^2 \varphi$ is a constant multiple of $L_2$. In particular, we ensure that \eqref{e:phisuper2p} holds if
\[ L_1 > C \left( \eps +  \|\nabla u\|_{L^\infty(Q_3)} \right)^{\gamma} L_2.\]
Therefore, for the case $\gamma > 0$, $\beta = 2$, and for any choice of $\eta>0$, using the comparison principle,
\[ 
\begin{split}
u(t,0)-u(t_0,0) 
\le \eta + C \left( \eps +  \|\nabla u\|_{L^\infty(Q_3)} \right)^{\gamma} \left( \eta^{-1} \|\nabla u\|_{L^\infty(Q_3)}^2  + \|u\|_{L^\infty(Q_3)} \right) (t-t_0).
\end{split}
\]
Choosing $\eta = \left( \eps +  \|\nabla u\|_{L^\infty(Q_3)} \right)^{\gamma/2+1} (t-t_0)^{1/2}$, we obtain,
\[ 
\begin{aligned}
u(t,0)-u(t_0,0) \leq C & \left( \eps +  \|\nabla u\|_{L^\infty(Q_3)} \right)^{\gamma/2+1} (t-t_0)^{1/2} + \\
& C \left( \eps +  \|\nabla u\|_{L^\infty(Q_3)} \right)^{\gamma} \|u\|_{L^\infty(Q_3)} (t-t_0).
\end{aligned}
\]
This finishes the proof for $\gamma > 0$ as well.
\end{proof}

\section{H\"older estimates for the spatial gradients}\label{sec:holder gradient}

In this section, we assume that $\gamma>-1$ so that Corollary
\ref{cor:lip} and Lemma \ref{lem:holder in t} holds, that is, the
solution of \eqref{eq:main va} in $Q_2$ has uniform interior Lipschitz
estimates in $x$ and uniform interior H\"older estimates in $t$, both
of which are independent of $\va\in (0,1)$. For $\rho, r>0$, we denote
\[
Q_r=B_r\times(-r^2,0],\quad Q^\rho_r=B_r\times (-\rho^{-\gamma}r^2,0].
\]
The cylinders $Q_r^\rho$ are the natural ones that correspond to the two-parameter family of scaling of the equation. Indeed, if $u$ solves \eqref{eq:main va} in $Q_r^\rho$ and
we let $v(x,t)=\frac{1}{r\rho}u(rx, r^2\rho^{-\gamma}t)$, then
\[
v_t(t,x)=\left(|\nabla v|^2+\va^2\rho^{-2}\right)^{\gamma/2}
\left(\Delta v+(p-2)\frac{v_iv_j}{|\nabla v|^2+\va^2 \rho^{-2}}
  v_{ij}\right)\quad\mbox{in } Q_1.
\]

If we choose
$\rho \geq \|\nabla u\|_{L^\infty(Q_{1})}+1$, we may assume that the solution
of \eqref{eq:main va} satisfies $|\nabla u|\le 1$ in $Q_1$.

We are going to show that $\nabla u$ is H\"older continuous in
space-time at the point $(0,0)$. The idea of the proof in this step is
similar to that in \cite{JS}. First we show that if the projection of
$\nabla u$ onto the direction $e\in\mathbb S^{n-1}$ is away from $1$
in a positive portion of $Q_1$, then $\nabla u\cdot e$ has improved
oscillation in a smaller cylinder.

\begin{lemma}\label{lem:osc2}
  Let $u$ be a smooth solution of \eqref{eq:main va} with
  $\va\in (0,1)$ such that $|\nabla u|\le 1$ in $Q_1$.  For every
  $\frac 12<\ell<1$, $\mu>0$, there exists $\tau_1\in (0,\frac 14)$
  depending only on $\mu, n$, and there exist $\tau, \delta>0$
  depending only on $n,p,\gamma, \mu$ and $\ell$ such that for
  arbitrary $e\in\mathbb{S}^{n-1}$, if
\begin{equation}\label{eq:proj on e positive}
|\{(x,t)\in Q_1: \nabla u\cdot e\le \ell\}|> \mu |Q_1|,
\end{equation}
then
\[
\nabla u\cdot e< 1-\delta\quad\mbox{in }Q^{1-\delta}_{\tau},
\]
and $Q^{1-\delta}_{\tau}\subset Q_{\tau_1}$.
\end{lemma}
\begin{proof}
Let
\begin{equation}\label{eq:coefficients}
a_{ij}(q)=(|q|^2+\va^2)^{\gamma/2} \left(\delta_{ij}+(p-2)\frac{q_iq_j}{|q|^2+\va^2}\right),\ q\in\R^n
\end{equation}
and denote
\[
a_{ij,m}=\frac{\partial a_{ij}}{\partial q_m}.
\]
Differentiating \eqref{eq:main va} in $x_k$, we have
\[
(u_k)_t=a_{ij}\big(u_{k})_{ij}+a_{ij,m}u_{ij} (u_{k})_m.
\]
Then 
\[
(\nabla u\cdot e-\ell)_t=a_{ij}\big(\nabla u\cdot
e-\ell)_{ij}+a_{ij,m}u_{ij} (\nabla u\cdot e-\ell)_m
\]
and for
\[
v=|\nabla u|^2,
\]
we have
\[
v_t=a_{ij}v_{ij}+a_{ij,m}u_{ij}v_m-2a_{ij}u_{ki}u_{kj}.
\] 
For $\rho=\ell/4$, let
\[
 w=(\nabla u\cdot e-\ell+\rho |\nabla u|^2)^+.
\]
Then in the region $\Omega_+=\{(x,t)\in Q_1: w>0\}$, we have
\[
 w_t=a_{ij} w_{ij} + a_{ij,m}u_{ij}w_m-2\rho a_{ij}u_{ki}u_{kj}.
\]
Since $|\nabla u|>\ell/2$ in $\Omega_+$, we have in $\Omega_+$:
\[
 |a_{ij,m}|\le 
 \begin{cases}
c(p,n,\gamma)\ell^{-1}, & \text{ if } \gamma\ge 0\\
c(p,n,\gamma)\ell^{\gamma-1}, & \text{ if } \gamma<0,
 \end{cases}
\]
where $c(p,n,\gamma)$ is a positive constant depending only on $p,n$ and $\gamma$. By Cauchy-Schwarz inequality, it follows that
\[
  w_t\le a_{ij} w_{ij} + c_1(\ell)|\nabla w|^2\quad\mbox{in }\Omega_+,
\]
where
\[
c_1(\ell)= 
 \begin{cases}
 c_0\ell^{-\gamma-3}, & \text{ if} \gamma\ge 0\\
c_0\ell^{2\gamma-3}, & \text{ if } \gamma<0
 \end{cases}
\]
for some constant $c_0>0$ depending only on $p,\gamma,n$. Therefore, it satisfies in the viscosity sense that
\[
  w_t\le \tilde a_{ij} w_{ij} + c_1(\ell)|\nabla w|^2\quad\mbox{in }Q_1,
\]
where
\[
\tilde a_{ij}(x)=
\begin{cases}
a_{ij}(\nabla u(x)),\quad x\in\Omega_+\\
\delta_{ij},\quad \mbox{elsewhere}.
 \end{cases}
\]
Notice that since $\ell\in (\frac 12, 1)$, $\tilde a_{ij}$ is
uniformly elliptic with ellipticity constants depending only on $p$
and $\gamma$. We can choose $c_2(\ell)>0$ depending only on
$p,\gamma,n$ and $\ell$ such that if we let
\[
 W=1-\ell+\rho
\]
and
\[
\overline w =\frac{1}{c_2}(1-e^{c_2(w -W )}),
\]
then we have
\[
   \overline w_t\ge \tilde a_{ij} \overline w_{ij} \quad\mbox{in }Q_1
\]
in the viscosity sense. Since $W\ge \sup_{Q_1}w$, then $\overline w\ge 0$ in $Q_1$. 

If $\nabla u\cdot e\le \ell$, then $\overline w\ge (1-e^{c_2 (\ell-1)})/c_2$. Therefore, it follows from the assumption that
\[
|\{(x,t)\in Q_1: \overline w\ge (1-e^{c_2 (\ell-1)})/c_2\}|> \mu |Q_1|.
\]
By Proposition 2.3 in \cite{JS}, there exist $\tau_1>0$ depending only $\mu$ and $n$, and $\nu>0$ depending only on $\mu, \ell, n,\gamma$ and $p$ such that
\[
\overline w\ge \nu\quad\mbox{in }Q_{\tau_1}.
\]
Meanwhile, we have
\[
\overline w\le W-w.
\]
This implies that
\[
W-w\ge\nu\quad\mbox{in }Q_{\tau_1}.
\]
Therefore, we have
\[
 \nabla u\cdot e+\rho|\nabla u|^2 \le 1+\rho -\nu\quad\mbox{in }Q_{\tau_1}.
\]
Since $| \nabla u\cdot e|\le |\nabla u|$, we have
\[
 \nabla u\cdot e+\rho ( \nabla u\cdot e)^2\le 1+\rho - \nu\quad\mbox{in }Q_{\tau_1}.
\]
Therefore, remarking that $\nu \le 1 + \rho$, 
%\textcolor{blue}{(since $\nu \le \bar w \le W \le 1+\rho$ in $Q_{\tau_1}$)}
we have
\[
\nabla u\cdot e \le \frac{-1+\sqrt{1+4\rho (1+\rho - \nu)}}{2\rho}\le
1-\delta\quad\mbox{in }Q_{\tau_1}
\]
for some $\delta>0$ depending only on $p, \gamma, \mu, \ell, n$.
Finally, we can choose $\tau=\tau_1$ if $\gamma<0$ and
$\tau=\tau_1(1-\delta)^{\gamma/2}$ if $\gamma\ge 0$ such that
$Q^{1-\delta}_\tau\subset Q_{\tau_1}$.
\end{proof}

Note that our choice of $\tau$ and $\delta$ in the above implies that
\[
\tau<(1-\delta)^{\frac{\gamma}{2}}\quad\mbox{when }\gamma\ge 0.
\]
In the rest of the paper, we will choose $\tau$ even smaller such that
%\begin{equation}\label{eq:relationtaudelta2}
%\tau<(1-\delta)^{1+\frac{\gamma}{2}}\quad\mbox{when }\gamma\ge 0.
%\end{equation}
%and
\begin{equation}\label{eq:relationtaudelta3}
\tau<(1-\delta)^{1+\gamma}\quad\mbox{for all }\gamma> -1. 
\end{equation}
%\textcolor{red}{\eqref{eq:relationtaudelta3} is stronger than \eqref{eq:relationtaudelta2}. \eqref{eq:relationtaudelta2} is written for our purpose and will be deleted later}. 
This fact will be used in the proof of Theorem \ref{thm:holdergradient}.

In case we can apply the previous lemma holds in all directions $e \in \partial B_1$, then it effectively implies a reduction in the oscillation of $\nabla u$ in a smaller parabolic cylinder. If such improvement of oscillation takes place at all scales, it leads to the H\"older continuity of
 $\nabla u$ at $(0,0)$ by iteration and scaling. The following corollary describes this favorable case in which the assumption of the previous Lemma holds in all directions.

%-------------------------------------------------------------------------------------
\begin{corollary}\label{cor:osc}
  Let $u$ be a smooth solution of \eqref{eq:main va} with
  $\va\in (0,1)$ such that $|\nabla u|\le 1$ in $Q_1$. For every
  $0<\ell<1$, $\mu>0$, there exist $\tau\in (0,1/4)$ depending only on
  $\mu$ and $n$, and $\delta>0$ depending only on
  $n,p,\gamma, \mu, \ell$, such that for every nonnegative integer
  $k\le\log\va/\log(1-\delta)$, if
\begin{equation}\label{eq:stopping condition}
  |\{(x,t)\in Q^{(1-\delta)^i}_{\tau^i}: \nabla u\cdot e\le \ell (1-\delta)^i\}|> \mu |Q^{(1-\delta)^i}_{\tau^i}|
  \quad\mbox{for all }e\in\mathbb S^{n-1}\mbox{ and }i=0,\cdots, k,
\end{equation}
then
\[
|\nabla u|< (1-\delta)^{i+1}\quad\mbox{in }Q^{(1-\delta)^{i+1}}_{\tau^{i+1}} \ \mbox{  for all }i=0,\cdots, k.
\]
\end{corollary}
%-------------------------------------------------------------------------------------
\begin{remark}
Remark that we can further impose on $\delta$ that $\delta < 1/2$ and $\delta < 1 -\tau$. 
\end{remark}
%---------------
\begin{proof}
  When $i=0$, it follows from Lemma \ref{lem:osc2} that
  $\nabla u\cdot e< 1-\delta$ in $Q_{\tau}$ for all
  $e\in\mathbb S^{n-1}$. This implies that $|\nabla u|<1-\delta$ in
  $Q^{1-\delta}_\tau$.

  Suppose this corollary holds for $i=0,\cdots, k-1$. We are going
  prove it for $i=k$. Let
\[
v(x,t)=\frac{1}{\tau^k (1-\delta)^k}u\Big(\tau^k x,
\tau^{2k}(1-\delta)^{-k\gamma} t\Big).
\] Then $v$ satisfies
\[
v_t=\left(|\nabla
  v|^2+\frac{\va^2}{(1-\delta)^{2k}}\right)^{\gamma/2}\left(\Delta
  v+(p-2)\frac{v_iv_j}{|\nabla v|^2+\eps^2 (1-\delta)^{-2k}}
  v_{ij}\right)\quad\mbox{in } Q_1.
\]
By the induction hypothesis, we also know that $|\nabla v|\le 1$ in $Q_1$, and
\[
|\{(x,t)\in Q_{1}: \nabla v\cdot e\le \ell \}|> \mu |Q_{1}| \quad\mbox{for all }e\in\mathbb S^{n-1}.
\]
Notice that $\va\le (1-\delta)^k$. Therefore, by Lemma \ref{lem:osc2} we have
\[
\nabla v\cdot e \le 1-\delta\quad\mbox{in }Q^{1-\delta}_{\tau} \quad\mbox{for all }e\in\mathbb S^{n-1}.
\]
Hence, $|\nabla v|\le 1-\delta$ in $Q^{1-\delta}_\tau$. Consequently,
\[
|\nabla u|< (1-\delta)^{k+1}\quad\mbox{in }Q^{(1-\delta)^{k+1}}_{\tau^{k+1}}. \qedhere
\]
\end{proof}

Unless $\nabla u(0,0)=0$, the above iteration will inevitably stop at some step. There will be a first value of $k$ where the assumptions of Corollary \ref{cor:osc} do not hold  in some direction $e\in\mathbb S^{n-1}$. This means that $\nabla u$ is close to some fixed vector in a large portion of $Q^{(1-\delta)^k}_{\tau^k}$. We then prove that $u$ is close to some linear function, from which the H\"older continuity of $\nabla u$ will follow applying a result from \cite{wang}. 

Having $\nabla u$ close to a vector $e$ for most points tells us that for every fixed time $t$, the function $u(x,t)$ will be approximately linear. However, it does not say anything about how $u$ varies respect to time. We must use the equation in order to prove that the function $u(x,t)$ will be close to some linear function uniformly in $t$. That is the main purpose of the following set of lemmas.

%-------------------------------------------------------------------------------------

\begin{lemma}\label{l:timeslices2allcylinder}
 Let $u\in C(\overline Q_1)$ be a smooth solution of \eqref{eq:main va} with $\gamma>-1, \va\in (0,1)$ such that $|\nabla u|\le M$ in $Q_1$. Let $A$ be a positive constant.
Assume that for all $t \in [-1,0]$, we have
\[ \osc_{B_1} u(\cdot, t) \leq A,\]
then
\[ 
\osc_{Q_1} u \leq 
\begin{cases}
C A,\quad\mbox{if }\gamma\ge 0\\
C (A+A^{1+\gamma})\quad\mbox{if }-1<\gamma<0,
\end{cases}
\]
where $C$ is a positive constant  depending only on $M, \gamma, p$ and the dimension $n$.\end{lemma}

\begin{proof}
When $\gamma\ge 0$, for the $a_{ij}$ in \eqref{eq:coefficients}, we have $|a_{ij}|\le \Lambda:=(M^2+1)^{\gamma/2}\max(p-1,1)$, and therefore, the conclusion follows from the same proof of Lemma 4.3 in \cite{JS}.

When $\gamma\in (-1, 0)$, we choose different comparison functions
from \cite{JS}. Let
\begin{eqnarray*}
\overline w(x,t) &= \overline a+ \Lambda A^{1+\gamma} t + 2 A |x|^\beta, \\
\underline w (x,t) &= \underline a - \Lambda A^{1+\gamma} t - 2 A |x|^\beta
\end{eqnarray*}
where $\beta=\frac{2+\gamma}{1+\gamma}$ and $\Lambda$ to be fixed
later. As far as $\overline a$ and $\underline a$ are concerned,
$\overline a$ is chosen so that
$\overline w(\cdot,-1) \geq u(\cdot,-1)$ in $B_1$ and
$\overline w(\bar x, -1) = u(\bar x,-1)$ for some
$\bar x\in\overline B_1$, and $\underline a$ is chosen so that
$\underline w(\cdot,-1) \leq u(\cdot,-1)$ in $B_1$ and
$\underline w(\underline x,-1) = u(\underline x,-1)$ for some
$\underline x\in \overline B_1$.  This implies that
\[
\overline a-\underline a=u(\bar x,-1) -u(\underline x,-1)+2\Lambda A^{1+\gamma}-2A|\bar x|^2-2A|\underline x|^2\le A+2\Lambda A^{1+\gamma}.
\]

Notice that $\beta>2$ since $\gamma\in (-1,0)$.  We now remark that if $\Lambda$ is chosen as follows: $\Lambda=(2\beta)^{\gamma+1}(\beta-1)pn^2+1$ 
then the following first inequality
\[
\Lambda A^{1+\gamma}\le
\Big((2A\beta|x|^{\beta-1})^2+\va^2\Big)^{\gamma/2} \cdot p n^2
\cdot 2A\beta(\beta-1)|x|^{\beta-2}\le
(2\beta)^{\gamma+1}(\beta-1)pn^2 A^{1+\gamma},
\]
(we used that $\gamma<0$) cannot hold true for $x \in B_1$. This
implies that $\overline w$ is a strict super-solution of the equation
satisfied by $u$. Similarly, $\underline w$ is a strict sub-solution.

We claim that 
\[
\overline w \geq u \quad\text{in } Q_1
\quad \text{ and } \underline w \leq u \quad\text{in } Q_1.
\]
We only justify the first inequality since we can proceed similarly to
get the second one.  If not, let $m=-\inf_{Q_1} (\overline w-u)>0$ and
$(x_0,t_0)\in\overline Q_1$ be such that
$m=u(x_0,t_0)-\overline w(x_0,t_0)$. Then $\overline w+m\ge u$ in $Q_1$ and
$\overline w(x_0,t_0)+m=u(x_0,t_0)$.  By the choice of $\bar a$, we
know that $t_0>-1$.  If $x_0\in\partial B_1$, then
\[
2A= (\overline w(x_0,t_0)+m)-(\overline w(0,t_0)+m)\le u(x_0,t_0)-u(0,t_0)\le \osc_{B_1} u(\cdot, t_0) \leq A,
\]
which is impossible. Therefore, $ x_0\in B_1$.  But this is not possible since $\overline w$ is a strict super-solution of the equation
satisfied by $u$. This proves the claim.

Therefore, we have
\[
\osc_{Q_1} u\le \sup_{Q_1}\overline w-\inf_{Q_1}\underline w\le \bar a-\underline a+4A=2\Lambda A^{\gamma+1}+5A. \qedhere
\]
\end{proof}

%-------------------------------------------------------------------------------------

\begin{lemma} \label{l:timeslices2allcylinder-aroundplane}
Let $u\in C(\overline Q_1)$ be a smooth solution of \eqref{eq:main va} with $\gamma\in\R, \va\in (0,1)$. Let $e\in\mathbb S^{n-1}$  and $0< \delta <1/8$.  Assume that for all $t \in [-1,0]$, we have
\[ \osc_{x\in B_1} (u(x, t) - x \cdot e) \leq \delta,\]
then
\[ 
\osc_{(x,t) \in Q_1} \left( u(x,t)  - x \cdot e \right)  \leq C \delta,
\]
where $C$ is a positive constant  depending only on $ \gamma, p$ and the dimension $n$.\end{lemma}

\begin{proof}
Let
\begin{eqnarray*}
\overline w(x,t) &= \overline a+ x \cdot e + \Lambda \delta t + 2 \delta |x|^2, \\
\underline w (x,t) &= \underline a + x \cdot e - \Lambda \delta t - 2 \delta |x|^2,
\end{eqnarray*}
where $\Lambda>0$ will be fixed later, $\overline a$ is chosen so that
$\overline w(\cdot,-1) \geq u(\cdot,-1)$ in $B_1$ and
$\overline w(\bar x, -1) = u(\bar x,-1)$ for some
$\bar x\in\overline B_1$, and $\underline a$ is chosen so that
$\underline w(\cdot,-1) \leq u(\cdot,-1)$ in $B_1$ and
$\underline w(\underline x,-1) = u(\underline x,-1)$ for some
$\underline x\in \overline B_1$. This implies that
\[
\overline a-\underline a=u(\bar x,-1)-\bar x\cdot e -(u(\underline x,-1)-\underline x\cdot e)+2\Lambda\delta-2\delta|\bar x|^2-2\delta|\underline x|^2\le (2\Lambda+1)\delta.
\]

For every $x \in \overline B_1$, and $t \in [-1,0]$, since $\delta < 1/8$, we have
\[ |\nabla \overline w(x,t)| \geq |e| - 4\delta |x| \geq 1/2, \qquad |\nabla \underline w(x,t)| \geq |e| - 4\delta |x| \geq 1/2.\]
Similarly, $|\nabla \overline w(x,t)| \leq 3/2$ and $|\nabla \underline w(x,t)| \leq 3/2$. Therefore, using the notation \eqref{eq:coefficients}, there is a constant $A_0$ (depending on $p$ and $\gamma$) so that 
\[ a_{ij}(\nabla \overline w(x,t)) \leq A_0 \mathrm I \qquad  \text{and} \qquad a_{ij}(\nabla \underline w(x,t)) \leq A_0 \mathrm I.\]

We choose $\Lambda = 5nA_0$. We claim that 
\[
\overline w \geq u \quad\text{in } Q_1
\quad \text{ and } \underline w \leq u \quad\text{in } Q_1.
\]
We only justify the first inequality since we can proceed similarly to
get the second one.  If not, let $m=-\inf_{Q_1} (\overline w-u)>0$ and
$(x_0,t_0)\in\overline Q_1$ be such that
$m=u(x_0,t_0)-\overline w(x_0,t_0)$. Then $\overline w+m\ge u$ in $Q_1$ and
$\overline w(x_0,t_0)+m=u(x_0,t_0)$.  By the choice of $\bar a$, we
know that $t_0>-1$.  If $x_0\in\partial B_1$, then
\[
\begin{split}
2\delta&= (\overline w(x_0,t_0)+m)- x_0\cdot e-(\overline w(0,t_0)+m)\\
&\le u(x_0,t_0)- x_0\cdot e-u(0,t_0)\leq \osc_{x\in B_1} (u(x, t_0) - x \cdot e)\le \delta,
\end{split}
\]
which is impossible. Hence, $ x_0\in B_1$. Therefore, we have the classical relations:
\begin{align*}
u(x_0,t_0) &= \overline w(x_0,t_0)+m, \\
\nabla u(x_0,t_0) &= \nabla \overline w(x_0,t_0) \in \overline B_{3/2} \setminus B_{1/2}, \\
D^2 u(x_0,t_0) &\leq D^2  \overline w(x_0,t_0) = 4\delta \mathrm I, \\
\partial_t u(x_0,t_0) &\geq \partial_t \overline w(x_0,t_0) = \Lambda \delta.
\end{align*}
It follows that
\[ u_t(x_0,t_0) - a_{ij}(\nabla u(x_0,t_0)) \partial_{ij} u(x_0,t_0) \geq \overline w_t(x_0,t_0) - a_{ij}(\nabla \overline w(x_0,t_0)) \partial_{ij} \overline w(x_0,t_0) > 0,\]
which is a contradiction. This proves the claim.

Therefore, we have
\[
\osc_{(x,t) \in Q_1} \left( u(x,t)  - x \cdot e \right) \le \sup_{Q_1}(\overline w-x\cdot e)-\inf_{Q_1}(\underline w-x\cdot e)\le \bar a-\underline a+4\delta=(2\Lambda +5)A. \qedhere
\]
\end{proof}

\begin{lemma} \label{l:gradientarounde2smalloscillation} Let $\eta$ be
  a positive constant and $u$ be a smooth solution of \eqref{eq:main
    va} with $\gamma>-1, \va\in (0,1)$ such that $|\nabla u|\le 1$ in
  $Q_1$. Assume
\[ \left\vert \left\{ (x,t) \in Q_1 : |\nabla u - e| > \eps_0 \right\}
\right\vert \leq \eps_1\]
for some $e\in\mathbb S^{n-1}$ and two positive constants
$\eps_0,\eps_1.$ Then, if $\eps_0$ and $\eps_1$ are sufficiently
small, there exists a constant $a \in \R$, such that
\[ |u(x,t) - a - e \cdot x| \leq \eta \quad \text{ for all } (x,t) \in Q_{1/2}.\]
Here, both $\va_0$ and $\va_2$ depend only on $n,p, \gamma$ and $\eta$.
\end{lemma}
%----------------------------------------------------------------------------------
\begin{proof}
Let $f(t):=| \{ x \in B_1 : |\nabla u(x,t) - e| > \eps_0 \} |$. By the assumptions and Fubini's theorem, we have that $\int_{-1}^0 f(t) d t\leq \eps_1$. It follows that for $E:=\{t\in (-1,0): f(t)\ge\sqrt{\eps_1}\}$, we obtain
\[
|E|\le \frac{1}{\sqrt{\eps_1}}\int _{E} f(t) d t\le \frac{1}{\sqrt{\eps_1}}\int _{-1}^0 f(t) d t\le \sqrt{\eps_1}.
\]
Therefore, for all $t \in (-1,0] \setminus E$, with $|E| \leq  \sqrt{\eps_1}$, we have
\begin{equation}\label{eq:aux-1}
 | \{ x \in B_1 : |\nabla u(x,t) - e| > \eps_0 \} | \leq \sqrt{\eps_1}.
 \end{equation}
It follows from \eqref{eq:aux-1} and Morrey's inequality that for all $t\in  (-1,0] \setminus E$, we have
\begin{equation}\label{eq:aux-2}
\osc_{B_{1/2}} (u(\cdot,t)-e\cdot x)\le C(n)\|\nabla u-e\|_{L^{2n}(B_1)}\le C(n)(\eps_0+\eps_1^{\frac{1}{4n}}),
\end{equation}
where $C(n)>0$ depends only on $n$. 

Meanwhile, since $|\nabla u| \leq 1$ in $Q_1$, we have that $\osc_{B_1} u(\cdot,t) \leq 2$ for all $t \in (-1,0]$. Thus, applying Lemma \ref{l:timeslices2allcylinder}, we have that $\osc_{Q_1} u \leq C$ for some constant $C$. Note that $u(t,x)-u(0,0)$ also satisfies \eqref{eq:main va} and $\|u(t,x)-u(0,0)\|_{L^\infty(Q_1)}\le \osc_{Q_1} u \leq C$. By applying Lemma \ref{lem:holder in t} to $u(t,x)-u(0,0)$, we have
\[\sup_{t \neq s, (t,x), (s,x)\in Q_1}
\frac{|u(t,x)-u(s,x)|}{|t-s|^{1/2}} \le
C.\]
Therefore, by \eqref{eq:aux-2} and the fact that $|E|\le\sqrt{\eps_1}$, we obtain 
\[
\osc_{B_{1/2}} (u(\cdot,t)-e\cdot x)\le C(\eps_0+\eps_1^{\frac{1}{4n}}+\eps_1^{\frac{1}{4}})
\]
for all $t \in (-1/4,0]$ (that is, including $t \in E$). If $\eps_0$ and $\eps_1$ are sufficiently small, we obtain from Lemma \ref{l:timeslices2allcylinder-aroundplane} that
\[
\osc_{Q_{1/2}}  (u-e\cdot x)\le C(\eps_0+\eps_1^{\frac{1}{4n}}+\eps_1^{\frac{1}{4}}).
\]
Hence, if $\eps_0$ and $\eps_1$ are sufficiently small, there exists a constant $a \in \R$, such that
\[ |u(t,x) - a - e \cdot x| \leq \eta \quad\text{ for all } (x,t) \in Q_{1/2}.\qedhere\]
\end{proof}
%----------------------------------------------------------------------------------
\begin{theorem}[Regularity of small perturbation
  solutions] \label{l:smallness to regularity} Let $u$ be a smooth
  solution of \eqref{eq:main va} in $Q_1$. For each $\beta\in (0,1)$,
  there exist two positive constants $\eta$ (small) and $C$ (large),
  both of which depends only on $\beta, n, \gamma$ and $p$, such that
  if $|u(x,t)-L(x)|\le\eta$ in $Q_1$ for some linear function $L$ of
  $x$ satisfying $1/2\le|\nabla L|\le 2$, then
\[
\|u-L\|_{C^{2,\beta}(Q_{1/2})}\le C.
\]
\end{theorem}
%----------------------------------------------------------------------------------
\begin{proof}
  Since $L$ is a solution of \eqref{eq:main va}, the conclusion
  follows from Corollary 1.2 in \cite{wang}.
\end{proof}

Now we are ready to prove the following H\"older gradient estimate.
%----------------------------------------------------------------------------------
\begin{theorem}\label{thm:holdergradient}
  Let $u$ be a smooth solution of \eqref{eq:main va} with
  $\va\in (0,1), \gamma>-1$ such that $|\nabla u|\le 1$ in $Q_1$.
  Then there exist two positive constants $\alpha$ and $C$ depending
  only on $n, \gamma$ and $p$ such that
\[
|\nabla u(x,t)-\nabla u(y,s)|\le C(|x-y|^\alpha+|t-s|^{\frac{\alpha}{2-\alpha\gamma}})
\]
for all $(x,t), (y,s)\in Q_{1/2}$. Also, there holds
\[
|u(x,t)-u(x,s)|\le C |t-s|^{\frac{1+\alpha}{2-\alpha\gamma}}
\]
for all $(x,t), (x,s)\in Q_{1/2}$. 
\end{theorem}
%----------------------------------------------------------------------------------
\begin{proof}
We first show the H\"older estimate of $\nabla u$ at $(0,0)$ and the H\"older estimate in $t$ at (0,0). 

Let $\eta$ be the one in Theorem \ref{l:smallness to regularity} with $\beta=1/2$, and for this $\eta$, let $\va_0,\va_1$ be two sufficiently small positive constants so that the conclusion of Lemma \ref{l:gradientarounde2smalloscillation} holds. For $\ell=1-\eps_0^2/2$ and $\mu=\va_1/|Q_1|$, if 
\[
|\{(x,t)\in Q_1: \nabla u\cdot e \le \ell\}|\le \mu |Q_1|\quad\mbox{for any }e\in\mathbb S^{n-1},
\]
then
\[ \left\vert \left\{ (x,t) \in Q_1 : |\nabla u - e| > \eps_0 \right\} \right\vert \leq \eps_1.\]
This is because if $|\nabla u(x,t)-e|>\eps_0$ for some $(x,t)\in Q_1$, then
\[
|\nabla u|^2-2\nabla u\cdot e+1\ge\eps_0^2.
\]
Since $|\nabla u|\le 1$, we have 
\[
\nabla u\cdot e\le 1-\eps_0^2/2.
\]
Therefore, if $\ell=1-\eps_0^2/2$ and $\mu=\va_1/|Q_1|$, then
\begin{equation}\label{eq:inclusion}
\left\{ (x,t) \in Q_1 : |\nabla u - e| > \eps_0 \right\}\subset  \{(x,t)\in Q_1: \nabla u\cdot e \le \ell\},
\end{equation}
from which it follows that
\[
\left\vert\left\{ (x,t) \in Q_1 : |\nabla u - e| > \eps_0 \right\}\right\vert \le  \left\vert\{(x,t)\in Q_1: \nabla u\cdot e\le \ell\}\right\vert\le \mu |Q_1|\le \eps_1.
\]

Let $\tau,\delta$ be the constants in Corollary \ref{cor:osc}. Denote $ [\log\va/\log(1-\delta)]$ as the integer part of $\log\va/\log(1-\delta)$. Let $k$ be either $[\log\va/\log(1-\delta)]$ or the minimum nonnegative integer such that the condition \eqref{eq:stopping condition} does not hold, whichever is smaller. Then it follows  from Corollary \ref{cor:osc} that for all $\ell=0,1,\cdots, k$, we have
\[
|\nabla u(x,t)|\le (1-\delta)^{\ell}\quad\mbox{in }Q^{(1-\delta)^{\ell}}_{\tau^{\ell}}.
\]
Then for $(x,t)\in Q^{(1-\delta)^{\ell}}_{\tau^{\ell}}\setminus Q^{(1-\delta)^{\ell+1}}_{\tau^{\ell+1}}$, %$\ell=0,1,\cdots, k-1$, 
\begin{equation}\label{eq:nabla-aux-1-bis}
|\nabla u(x,t)|\le (1-\delta)^{\ell}\le C (|x|^\alpha+|t|^{\frac{\alpha}{2-\alpha\gamma}}),
\end{equation}
where $C=\frac{1}{1-\delta}$ and $\alpha=\frac{\log(1-\delta)}{\log\tau}$. Thus, 
\begin{equation}\label{eq:nabla-aux-1}
|\nabla u(x,t)-q|\le C(|x|^\alpha+|t|^{\frac{\alpha}{2-\alpha\gamma}})\quad\mbox{in }Q_1\setminus Q^{(1-\delta)^{k+1}}_{\tau^{k+1}}
\end{equation}
for every $q\in\R^n$ such that $|q|\le (1-\delta)^k$. Note that when $\gamma\ge 0$, it follows from \eqref{eq:relationtaudelta3} that
\begin{equation}\label{eq:rangeofalpha}
2-\alpha\gamma>0\quad\mbox{and}\quad\frac{\alpha}{2-\alpha\gamma}< \frac 12.
\end{equation}
%This implies that $\alpha< \frac{2}{2+\gamma}$ when $\gamma\ge 0$. 
For $\ell=0,1,\cdots, k$, let 
\begin{equation}\label{eq:rescalingsolutions}
u_\ell(x,t)=\frac{1}{\tau^\ell (1-\delta)^\ell}u(\tau^\ell x, \tau^{2\ell}(1-\delta)^{-\ell\gamma} t).
\end{equation}
Then $|\nabla u_\ell(x,t)|\le 1$ in $Q_1$, and 
\begin{equation}\label{eq:quasilinear}
\partial_t u_\ell = (|\nabla u_\ell |^2+\va^2(1-\delta)^{-2\ell})^{\gamma/2} 
\left(\delta_{ij}+(p-2)\frac{\partial_i u_\ell \partial_j u_\ell}{|\nabla u_\ell|^2+\va^2(1-\delta)^{-2\ell}}\right)\partial_{ij}u_\ell\quad\mbox{in }Q_1.
\end{equation}
Notice that $\va^2(1-\delta)^{-2\ell}\le \va^2(1-\delta)^{-2k}\le 1$. By Lemma \ref{l:timeslices2allcylinder}, we have
\[
\osc_{Q_1} u_\ell \le C,
\]
and thus,
\begin{equation}\label{eq:for time1}
\osc_{Q^{(1-\delta)^{\ell}}_{\tau^{\ell}}} u \le C \tau^\ell (1-\delta)^\ell.
\end{equation}
Let $v=u_k$. 

Case 1: $k=[\log\va/\log(1-\delta)]$. Then
we have $(1-\delta)^{k+1}<\va\le(1-\delta)^k$, and thus,
$\frac 12<1-\delta<\va(1-\delta)^{-k}\le 1$. Therefore, when $\ell=k$,  the equation
\eqref{eq:quasilinear} is a uniformly parabolic quasilinear equation
with smooth and bounded coefficients. By the standard quasilinear
parabolic equation theory (see, e.g., Theorem 4.4 of \cite{LSU} in page 560) and Schauder estimates, there exists
$b\in \R^n, |b|\le 1$ such that
\[
|\nabla v(x,t)-b|\le C(|x|+|t|^{1/2})\le C(|x|^\alpha+|t|^{\frac{\alpha}{2-\alpha\gamma}})\quad\mbox{in }Q_{\tau}^{1-\delta}\subset Q_{1/4}
\]
and
\[
|\partial_t v|\le C\quad\mbox{in }Q_{\tau}^{1-\delta}\subset Q_{1/4},
\]
where $C>0$ depends only on $\gamma, p$ and $n$, and we used that $\frac{\alpha}{2-\alpha\gamma}\le \frac 12$. Rescaling back, we have
\begin{equation}\label{eq:nabla-aux-2}
|\nabla u(x,t)-(1-\delta)^kb|\le C(|x|^\alpha+|t|^{\frac{\alpha}{2-\alpha\gamma}})\quad\mbox{in }Q^{(1-\delta)^{k+1}}_{\tau^{k+1}}
\end{equation}
and
\begin{equation}\label{eq:for time2}
|u(x,t)-u(x,0)|\le C \tau^{-k}(1-\delta)^{k(\gamma+1)}|t|\quad\mbox{in }Q^{(1-\delta)^{k+1}}_{\tau^{k+1}}.
\end{equation}
Then we can conclude from \eqref{eq:nabla-aux-1} and \eqref{eq:nabla-aux-2} that
\[
|\nabla u(x,t)-q|\le C(|x|^\alpha+|t|^{\frac{\alpha}{2-\alpha\gamma}})\quad\mbox{in }Q_{1/2},
\]
where $C>0$ depends only on $\gamma, p$ and $n$. From \eqref{eq:for time2}, we obtain that for $|t|\le \tau^{2m}(1-\delta)^{-m\gamma}$ with $m\ge k+1$, 
\begin{equation}\label{eq:for time3}
|u(0,t)-u(0,0)|\le C\tau^{-k}(1-\delta)^{k(\gamma+1)}\tau^{2m}(1-\delta)^{-m\gamma}\le C\tau^{m}(1-\delta)^{m},
\end{equation}
where in the last inequality we have used \eqref{eq:relationtaudelta3}. From \eqref{eq:for time1} and \eqref{eq:for time3}, we have
\[
|u(0,t)-u(0,0)|\le C|t|^\beta,
\] 
for all $t\in (-1/4,0]$, where $\beta$ is chosen such that
\[
\tau(1-\delta)=(\tau^2(1-\delta)^{-\gamma})^\beta.
\]
That is, 
\begin{equation}\label{eq:holderexponentint}
\beta=\frac{1+\alpha}{2-\alpha\gamma}.
\end{equation}
Note that $\beta>\frac 12$ if $\gamma>-2$.

Case 2:  $k<[\log\va/\log(1-\delta)]$. Then
\[
|\{(x,t)\in Q^{(1-\delta)^{k}}_{\tau^k}: \nabla u\cdot e\le \ell (1-\delta)^k\}|\le \mu |Q^{(1-\delta)^{k}}_{\tau^k}|\quad\mbox{for some }e\in\mathbb S^{n-1}.
\]
Also,
\[
|\nabla u|< (1-\delta)^{\ell}\quad\mbox{in }Q^{(1-\delta)^{\ell}}_{\tau^{\ell}} \mbox{ for all }\ell=0,1,\cdots, k.
\]
Recall $v=u_k$ as defined in \eqref{eq:rescalingsolutions}, which satisfies \eqref{eq:quasilinear} with $\ell=k$. Then $|\nabla v|\le 1$ in $Q_1$, and
\[
|\{(x,t)\in Q_{1}: \nabla v\cdot e\le \ell\}|\le \mu |Q_{1}|\quad\mbox{for some }e\in\mathbb S^{n-1}.
\]
Consequently, using \eqref{eq:inclusion}, we get
\[ \left\vert \left\{ (x,t) \in Q_1 : |\nabla v - e| > \eps_0 \right\} \right\vert \leq \eps_1.\]
 It follows from Lemma \ref{l:gradientarounde2smalloscillation} that there exists $a\in\R$ such that
\[ |v(x,t) - a - e \cdot x| \leq \eta \quad \text{ for all } (x,t) \in Q_{1/2}.\]
By Theorem \ref{l:smallness to regularity}, there exists $b\in\R^n$ such that
\[
|\nabla v-b|\le C (|x|+\sqrt{|t|})  \quad\mbox{for all }(x,t)\in Q_{\tau}^{1-\delta}\subset Q_{1/4}.
\]
and
\[
|\partial_t v|\le C\quad\mbox{in }Q_{\tau}^{1-\delta}\subset Q_{1/4}.
\]
Rescaling back, we have
\[
|\nabla u(x,t)-(1-\delta)^kb|\le C(|x|^\alpha+|t|^{\frac{\alpha}{2-\alpha\gamma}})\quad\mbox{in }Q^{(1-\delta)^{k+1}}_{\tau^{k+1}}
\]
and
\[
|u(x,t)-u(x,0)|\le C \tau^{-k}(1-\delta)^{k(\gamma+1)}|t|\quad\mbox{in }Q^{(1-\delta)^{k+1}}_{\tau^{k+1}}.
\]
Together with \eqref{eq:nabla-aux-1} and \eqref{eq:for time1},  we can conclude as in Case 1 that
\[
|\nabla u(x,t)-q|\le C(|x|^\alpha+|t|^{\frac{\alpha}{2-\alpha\gamma}})\quad\mbox{in }Q_{1/2},
\]
and
\[
|u(0,t)-u(0,0)|\le C|t|^\beta,
\] 
for all $t\in (-1/4,0]$, where $C>0$ depends only on $\gamma, p$ and $n$.

In conclusion, we have proved that there exist $q\in\R^n$ with
$|q|\le 1$, and two positive constants $\alpha, C$ depending only on
$\gamma, p$ and $n$ such that
\[
|\nabla u(x,t)-q|\le C(|x|^\alpha+|t|^{\frac{\alpha}{2-\alpha\gamma}}) \quad\mbox{for all }(x,t)\in Q_{1/2}
\]
and
\[
|u(0,t)-u(0,0)|\le C|t|^\beta,\quad\mbox{for }t\in(-1/4,0],
\] 
where $\beta$ is given in \eqref{eq:holderexponentint}. Then the conclusion follows from standard translation arguments.
\end{proof}

%Note that the assumption \eqref{eq:relationtaudelta3} implies the H\"older exponent
%$
%\alpha<\frac{1}{1+\gamma}.
%$

\section{Approximation}\label{sec:approximation}
As mentioned in the introduction, the viscosity solutions to 
\begin{equation}\label{eq:degenerate2}
u_t=|\nabla u|^{\gamma}\left(\Delta u + (p-2)|\nabla u|^{-2}u_iu_ju_{ij}\right)\quad\mbox{in }Q_1
\end{equation}
with $\gamma>-1$ and $p>1$ fall into the general framework studied by
Ohnuma-Sato in \cite{OS}, which is an extension of the work of
Barles-Georgelin \cite{BGe} and Ishii-Souganidis \cite{IS} on the viscosity solutions of
singular/degenerate parabolic equations. Let us recall the definition
of viscosity solutions to \eqref{eq:degenerate2} in \cite{OS}.

We denote 
\[
F(\nabla u,\nabla^2 u)=|\nabla u|^{\gamma}\left(\Delta u + (p-2)|\nabla u|^{-2}u_iu_ju_{ij}\right).
\]

Let $\mathcal F$ be the set of functions $f\in C^2([0,\infty))$ satisfying
\[
f(0)=f'(0)=f''(0)=0,\quad f''(r)>0\mbox{ for all }r>0,
\]
and
\[
\lim_{|x|\to0,x\neq 0}F(\nabla g(x),\nabla^2 g(x))=\lim_{|x|\to0,x\neq 0}F(-\nabla g(x),-\nabla^2 g(x))=0, \quad\mbox{where }g(x)=f(|x|).
\]
This set $\mathcal F$ is not empty when $\gamma>-1$ and $p>1$, since
$f(r)=r^\beta\in\mathcal F$ for any
$\beta>\max(\frac{\gamma+2}{\gamma+1},2)$. Moreover, if
$f\in\mathcal F$, then $\lambda f\in\mathcal F$ for all $\lambda>0$.

Because the equation \eqref{eq:degenerate2} may be singular or
degenerate, one needs to choose the test functions properly when
defining  viscosity solutions. A function $\varphi\in C^2(Q_1)$ is
admissible, which is denoted as $\varphi\in\mathcal A$, if for every
$\hat z=(\hat z, \hat t)\in Q_1$ that $\nabla \varphi(\hat z)=0$,
there exist $\delta>0$, $f\in\mathcal F$ and $\omega\in C([0,\infty))$
satisfying $\omega\ge 0$ and $\lim_{r\to 0}\frac{\omega(r)}{r}=0$ such
that for all $z=(x,t)\in Q_1, |z-\hat z|<\delta$ we have
\[
|\varphi(z)-\varphi(\hat z)-\varphi_t(\hat z)(t-\hat t)|\le f(|x-\hat x|)+\omega(|t-\hat t|).
\] 
\begin{definition}
  An upper (lower, resp.) semi-continuous function $u$ in $Q_1$ is
  called a viscosity subsolution (supersolution, resp.) of
  \eqref{eq:degenerate2} if for every $\varphi\in C^2(Q_1)$,
  $u-\varphi$ has a local maximum (minimum, resp.) at
  $(x_0,t_0)\in Q_1$, then
\[
\varphi_t\le (\ge, resp.)|\nabla \varphi|^{\gamma} \left(\Delta
  \varphi + (p-2)|\nabla
  \varphi|^{-2}\varphi_i\varphi_j\varphi_{ij}\right) \mbox{ at
}(x_0,t_0)\mbox{ when }\nabla \varphi(x_0,t_0)\neq 0
\]
and
\[
\varphi_t\le (\ge, resp.)\ 0 \mbox{ at }(x_0,t_0)\mbox{ when }\nabla \varphi(x_0,t_0)= 0.
\]

A function $u\in C(Q_1)$ is called a viscosity solution of
\eqref{eq:degenerate}, if it is both a viscosity subsolution and a
viscosity supersolution.
\end{definition}

We shall use two properties about the viscosity solutions defined in
the above. The first one is the comparison principle for
\eqref{eq:degenerate2}, which is Theorem 3.1 in \cite{OS}.
%--------------------------------------------------
\begin{theorem}[Comparison principle]\label{thm:comparison principle}
  Let $u$ and $v$ be a viscosity subsolution and a viscosity
  supersolution of \eqref{eq:degenerate2} in $Q_1$, respectively. If
  $u\le v$ on $\partial_pQ_1$, then $u\le v$ in $\overline Q_1$.
\end{theorem}
%--------------------------------------------------
The second one is the stability of viscosity solutions of
\eqref{eq:degenerate2}, which is an application of Theorem 6.1 in
\cite{OS}. Its application to the equation \eqref{eq:degenerate2} with
$\gamma=0,1<p\le 2$ is given in Proposition 6.2 in \cite{OS} with
detailed proof. It is elementary to check it applies to
\eqref{eq:degenerate2} for all $\gamma>-1$ and all $p>1$ (which was
also pointed out in \cite{OS}).
%--------------------------------------------------
\begin{theorem}[Stability]\label{thm:stability}
  Let $\{u_k\}$ be a sequence of bounded viscosity subsolutions of
  \eqref{eq:main va} in $Q_1$ with $\va_k\ge 0$ that $\va_k\to 0$, and
  $u_k$ converges locally uniformly to $u$ in $Q_1$. Then $u$ is a
  viscosity subsolution of \eqref{eq:degenerate2} in $Q_1$.
\end{theorem}
%--------------------------------------------------

Now we shall use the solution of \eqref{eq:main va} to approximate the
solution of \eqref{eq:degenerate2}. Since $p>1$, it follows from
classical quasilinear equations theory (see e.g.
\cite[Theorem~4.4,~p.~560]{LSU}) and the Schauder estimates that
%--------------------------------------------------
\begin{lemma}\label{lem:app1}
  Let $g\in C(\partial_p Q_1)$. For $\va>0$, there exists a unique
  solution $u^\va\in C^{\infty}(Q_1)\cap C(\overline Q_1)$ of
  \eqref{eq:main va} with $p>1$ and $\gamma\in\R$ such that $u^\va=g$
  on $\partial_p Q_1$.
\end{lemma}
%--------------------------------------------------
The last ingredient we need in the proof of Theorem
\ref{thm:mainholdergradient} is the following continuity estimate up
to the boundary for the solutions of \eqref{eq:main va}, where the
proof is given in the appendix. For two real numbers $a$ and $b$, we
denote $a\vee b=\max(a,b)$, $a\wedge b=\min(a,b)$.
%--------------------------------------------------
\begin{theorem}[Boundary estimates]\label{prop:boundary regularity}
  Let $u\in C(\overline Q_1)\cap C^\infty(Q_1)$ be a solution of
  \eqref{eq:main va} with $\gamma>-1$ and $\va\in(0,1)$. Let
  $\varphi:=u|_{\partial_p Q_1}$ and let $\rho$ be a modulus of
  continuity of $\varphi$. Then there exists another modulus of
  continuity $\rho^*$ depending only on
  $n,\gamma, p, \rho, \|\varphi\|_{L^\infty(\partial_p Q_1)}$ such
  that
\[
 |u(x,t)-u(y,s)|\le \rho^*(|x-y|\vee\sqrt{|t-s|})
\]
for all $(x,t), (y,s)\in \overline Q_1$.
\end{theorem}
%--------------------------------------------------
\begin{proof}[Proof of Theorem \ref{thm:mainholdergradient}]
  Given Theorem \ref{thm:holdergradient}, Theorem \ref{thm:comparison
    principle}, Theorem \ref{thm:stability}, Lemma \ref{lem:app1} and
  Theorem \ref{prop:boundary regularity}, the proof of Theorem
  \ref{thm:mainholdergradient} is identical to that of Theorem 1 in
  \cite{JS}.
\end{proof}

\appendix

\section{Appendix} \label{A}

We will adapt some arguments in \cite{CKLS} to prove Theorem
\ref{prop:boundary regularity}. In the following, $c$ denotes some
positive constant depending only on $n,\gamma$ and $p$, which may vary
from line to line. Denote
\[
F_\va(\nabla u,\nabla^2 u)= (|\nabla u |^2+\va^2)^{\gamma/2}
\left(\delta_{ij}+(p-2)\frac{u_iu_j}{|\nabla u|^2+\va^2}\right)u_{ij}.
\]
%--------------------------------------------------
\begin{lemma}\label{lem:sub-super-solutions}
  For every $z\in\partial B_1$, there exists a function
  $W_z\in C(\overline B_1)$ such that $W_z(z)=0, W_z>0$ in
  $\overline B_1\setminus\{z\}$, and
\[
F_\va(\nabla W_z,\nabla^2 W_z)\le -1\quad\mbox{in }B_1.
\]
\end{lemma}
%--------------------------------------------------
\begin{proof}
  Let $z\in \partial B_1$. Let $f(r)=\sqrt{(r-1)^+}$ and
  $w_z(x)=f(|x-2z|)$. Then for $x\in B_1$, we have
\[
F_\va(\nabla w_z, \nabla^2
w_z)=(f'^2+\va^2)^\frac{\gamma}{2}\left(\Big(1+(p-2)\frac{f'^2}{f'^2+\va^2}\Big)f''+\frac{n-1}{|x-2z|}f'\right).
\]
Then there exists $\delta>0$ depending only on $n,\gamma$ and $p$ such
that for $x\in B_1\cap B_{1+\delta}(2z)$, we have
\[
F_\va(\nabla w_z, \nabla^2 w_z)\le -1.
\]
For $\sigma=\frac{2n}{\min(p-1,1)}+2$ and $a>0$, let
$G_z(x)=a(2^\sigma-\frac{1}{|x-2z|^\sigma})$. Then
$G_z(x)\ge a(2^\sigma-1)$ in $B_1$. Also, for $r=|x-2z|$ and
$x\in B_1$, we have
\[
\begin{split}
  &F_\va(\nabla G_z, \nabla^2 G_z)\\
  &=a(\sigma^2r^{-2\sigma-2}+\va^2)^\frac{\gamma}{2}\left(\Big(1+\frac{(p-2)\sigma^2}{\sigma^2+\va^2r^{2\sigma+2}}\Big)
    \sigma(-\sigma-1)r^{-\sigma-2}+(n-1)\sigma r^{-\sigma-2}\right)\\
  &\le -\frac{a}{2}\sigma r^{-\sigma-2}(\sigma^2r^{-2\sigma-2}+\va^2)^\frac{\gamma}{2}\\
  &\le
\begin{cases}
 -\frac{a}{2}3^{-\sigma-2-\gamma(\sigma+1)}\sigma^{1+\gamma}\quad\mbox{when }\gamma\ge 0\\
  -\frac{a}{2} 3^{-\sigma-2}(\sigma^2+1)^{\gamma/2}\sigma\quad\mbox{when }\gamma< 0,\\
\end{cases}
\end{split}
\]
where in the first inequality we used the choice of $\sigma$. Then we choose $a$ that
\[
a(2^\sigma-\frac{1}{|1+\delta|^\sigma})=\sqrt{\delta/2}.
\]
Since $w_z(z)=0$ and $G_z(z)>0$, the function
\[
W_z(x)=
\begin{cases}
G_z(x)\quad\mbox{for }x\in \overline B_1,\ |x-2z|\ge 1+\delta\\
\min(G_z(x),w_z(x))\quad\mbox{for }x\in \overline B_1,\ |x-2z|\le 1+\delta\\
\end{cases}
\]
agrees with $w_z$ in a neighborhood of $z$ (relative to
$\overline B_1$). Also, because of the choice of $a$, $W_z$ agrees
with $G_z$ when $x\in \overline B_1$ and $|x-2z|\ge 1+\tilde\delta$
for some $\tilde\delta\in (0,\delta)$. Moreover,
\[
F_\va(\nabla W_z,\nabla^2 W_z)\le-\kappa
\]
for some constant $\kappa>0$ depending only on $n,\gamma$ and $p$. By
multiplying a large positive constant to $W_z$, we finish the proof of
this lemma.
\end{proof}
%--------------------------------------------------
\begin{lemma}\label{lem:parabolic-sub-super-solutions}
  For every $(z,\tau)\in\partial_p Q_1$, there exists
  $W_{z,\tau}\in C(\overline Q_1)$ such that $W_{z,\tau}(z,\tau)=0$,
  $W_{z,\tau}>0$ in $\overline Q_1\setminus\{(z,\tau)\}$, and
\[
\partial_t W_{z,\tau}-F_\va(\nabla W_{z,\tau},\nabla^2 W_{z,\tau})\ge 1\quad\mbox{in }Q_1.
\]
\end{lemma}
%--------------------------------------------------
\begin{proof}
For $\tau>-1$ and $x\in\partial B_1$, then
\[
W_{z,\tau}(x,t)=\frac{(t-\tau)^2}{2}+2W_z
\]
is a desired function, where $W_z$ is the one in Lemma
\ref{lem:sub-super-solutions}. For $\tau=-1$ and $x\in B_1$, we let
\[
W_{z,\tau}(x,t)=A(t+1)+|x-z|^\beta,
\]
where $\beta=\max(\frac{\gamma+2}{\gamma+1},2)$. Then if we choose
$A>0$ large, which depends only on $n,\gamma$ and $p$, then
$W_{z,\tau}$ will be a desired function.
\end{proof}
For two real numbers $a$ and $b$, we denote $a\vee b=\max(a,b)$,
$a\wedge b=\min(a,b)$.
%--------------------------------------------------
\begin{theorem}\label{prop:boundary regularity-aux}
  Let $u\in C(\overline Q_1)\cap C^\infty(Q_1)$ be a solution of
  \eqref{eq:main va} with $\gamma>-1$ and $\va\in(0,1)$. Let
  $\varphi:=u|_{\partial_p Q_1}$ and let $\rho$ be a modulus of
  continuity of $\varphi$. Then there exists another modulus of
  continuity $\rho^*$ depending only on $n,\gamma, p, \rho$ such that
\[
 |u(x,t)-u(y,s)|\le \tilde\rho(|x-y|\vee |t-s|)
\]
for all $(x,t)\in \overline Q_1,  (y,s)\in\partial_p Q_1$.
\end{theorem}
%--------------------------------------------------
\begin{proof}
For every $\kappa>0$ and $(z,\tau)\in\partial_p Q_1$, let
\[
W_{\kappa,z,\tau}(x,t)=\varphi(z,\tau)+\kappa+M_\kappa
W_{z,\tau}(x,t),
\]
where $M_\kappa>0$ is chose so that
\[
\varphi(z,\tau)+\kappa+M_\kappa W_{z,\tau}(y,s)\ge
\varphi(y,s)\quad\mbox{for all }(y,s)\in\partial_p Q_1.
\]
Indeed,
\[
M_k=\inf_{(y,s)\in\partial_p Q_1, (y,s)\neq(z,\tau)}\frac{ (\rho(|z-y|\vee |\tau-s|)-\kappa)^+}{W_{z,\tau}(y,s)}
\]
would suffice, and is independent of the choice of $(z,\tau)$. Finally, let
\[
W(x,t)=\inf_{\kappa>0,(z,\tau)\in\partial_p Q_1}W_{\kappa,z,\tau}(x,t).
\]
Note that for every $\kappa>0$ and $(z,\tau)\in\partial_p Q_1$,
\[
\begin{split}
W(x,t)-\varphi(z,\tau)&\le W_{\kappa,z,\tau}(x,t)-\varphi(z,\tau)\\
&\le \kappa+M_\kappa W_{z,\tau}(x,t)\\
&\le \kappa+M_\kappa (W_{z,\tau}(x,t)-W_{z,\tau}(z,\tau))\\
&\le \kappa+M_\kappa\omega(|z-x|\vee|\tau-t|),
\end{split}
\]
where $\omega$ is the modulus of continuity for $W_{z,\tau}$, which is
evidently independent of $(z,\tau)$. Let
$\tilde\rho(r)=\inf_{\kappa>0}(\kappa+M_\kappa\omega(r))$ for all
$r\ge 0$. Then $\tilde\rho$ is a modulus of continuity, and
\[
W(x,t)-\varphi(z,\tau)\le \tilde\rho(|z-x|\vee|\tau-t|)\quad\mbox{ for
  all }(x,t)\in \overline Q_1, (z,\tau)\in\partial_p Q_1.
\]
By Lemma \ref{lem:parabolic-sub-super-solutions}, $W_{\kappa,z,\tau}$
is a supersolution of \eqref{eq:main va} for every $\kappa>0$ and
$(z,\tau)\in\partial_p Q_1$, and therefore, $W$ is also a
supersolution of \eqref{eq:main va}. By the comparison principle,
\[
u(x,t)-\varphi(z,\tau)\le W(x,t)-\varphi(z,\tau)\le \tilde\rho(|z-x|\vee |\tau-t|)
\]
for all $(x,t)\in \overline Q_1,  (z,\tau)\in\partial_p Q_1$.

Similarly, one can show that
$u(x,t)-\varphi(z,\tau)\ge -\tilde\rho(|z-x|\vee |\tau-t|)$ for all
$(x,t)\in \overline Q_1, (z,\tau)\in\partial_p Q_1$. This finishes the
proof of this theorem.
\end{proof}
%--------------------------------------------------
\begin{proof}[Proof of Theorem \ref{prop:boundary regularity}]
By the maximum principle, we have that
\[
M:=\|u\|_{L^\infty(Q_1)}=\|\varphi\|_{L^\infty(\partial_p Q_1)}.
\]
Let $(x,t), (y,s)\in Q_1$, and we assume that $t\ge s$. Let $x_0$ be
such that $|x-x_0|=1-|x|=r$. Let $\tilde \rho$ be the one in the
conclusion of Theorem \ref{prop:boundary regularity-aux}. Without loss
of generality, we may assume that $2M+2\ge \tilde\rho(r)\ge r$ for all
$r\in[0,2]$ (e.g., replacing $\tilde\rho(r)$ by $\tilde\rho(r)+r$),
and $\tilde\rho(r)\le 2M+2$ for all $r\ge 2$.

In the following, if $\gamma\in(-1,0)$, then we will assume first that
\[
r^{1+\gamma}(2M+2)^{-\gamma}\le 1,
\]
and will deal with the other situation in the end of this proof. Under
the above assumption, we have that
$r^{2+\gamma} (\tilde\rho(2r))^{-\gamma}\le r^{2+\gamma}
(2M+2)^{-\gamma}\le r$
when $\gamma<$0, and
$r^{2+\gamma} (\tilde\rho(2r))^{-\gamma}\le r^{2+\gamma}
(\tilde\rho(r))^{-\gamma}\le r^2\le r$
when $\gamma\ge 0$. Thus, for all $\gamma>-1$, we have
\[
r^{2+\gamma} (\tilde\rho(2r))^{-\gamma}\le r.
\]
We will deal with the situation that $\gamma\in(-1,0)$ and
$r^{1+\gamma}(2M+2)^{-\gamma}\ge 1$ in the very end of the proof.
\medskip

\emph{Case 1:} $r^{2+\gamma} (\tilde\rho(2r))^{-\gamma}\le 1+t$.

\medskip

If $|y-x|\le r/2$ and $|s-t|\le r^{2+\gamma} (\tilde\rho(2r))^{-\gamma}/4$, then we do a scaling:
\[
v(z,\tau)=\frac{u(rz+x,r^{2+\gamma} (\tilde\rho(2r))^{-\gamma}\tau+t)-u(x_0,t)}{\tilde\rho(2r)}.
\]
Then 
\[
v_\tau=(|\nabla v |^2+\va^2r^2\tilde\rho(2r)^{-2})^{\gamma/2}
\left(\delta_{ij}+(p-2)\frac{v_iv_j}{|\nabla
    u|^2+\va^2r^2\tilde\rho(2r)^{-2}}\right)u_{ij}\quad\mbox{in }Q_1.
\]
Notice that $\va r /\tilde\rho(2r)\le \va r /\tilde\rho(r)\le \va<1$
and $r^{2+\gamma} (\tilde\rho(2r))^{-\gamma}\le r$. Thus,
$|v(z,\tau)|\le 1$ for $(z,\tau)\in Q_1$. Applying Corollary
\ref{cor:lip} and Lemma \ref{lem:holder in t} to $v$ and rescaling to
$u$, there exists $\alpha>0$ depending only on $\gamma$ such that $v$
is $C^\alpha$ in $(x,t)$, and there exists $C>0$ depending only on
$n,\gamma$ and $p$, such that
\[
|u(y,s)-u(x,s)|\le C\tilde\rho(2r)\frac{|x-y|^\alpha}{r^\alpha}
\]
and
\[
|u(x,t)-u(x,s)|\le C
\tilde\rho(2r)^{1+\alpha\gamma}\frac{|t-s|^{\alpha}}{r^{\alpha(2+\gamma)}},
\]
Therefore,
\[
|u(y,s)-u(x,t)|\le C\tilde\rho(2r)\frac{|x-y|^\alpha}{r^\alpha}+C
\tilde\rho(2r)^{1+\alpha\gamma}\frac{|t-s|^{\alpha}}{r^{\alpha(2+\gamma)}}.
\]
Since $|y-x|\le r/2$ and
$|s-t|\le r^{2+\gamma} (\tilde\rho(2r))^{-\gamma}/4\le r/4$, we have
$2^{-m-1}r<|x-y|\vee|t-s|\le 2^{-m}r$ for some integer $m\ge 1$. Then
\[
\begin{split}
  |u(y,s)-u(x,t)|&\le C\frac{\tilde\rho(2^{m+2}
    (|x-y|\vee|t-s|))}{2^{m\alpha}}
  +C \frac{\tilde\rho(2^{m+2} (|x-y|\vee|t-s|))^{1+\alpha\gamma}}{2^{m\alpha}r^{\alpha(1+\gamma)}}\\
  &\le C\frac{\tilde\rho(2^{m+2} (|x-y|\vee|t-s|))+\tilde\rho(2^{m+2}
    (|x-y|\vee|t-s|))^{1+\alpha\gamma}}{2^{m\alpha}}.
\end{split}
\]
Notice that
\[
\sup_{m\ge 1}\frac{\tilde\rho(2^{m+2} r)+\tilde\rho(2^{m+2} r)^{1+\alpha\gamma}}{2^{m\alpha}}\to 0\quad\mbox{as }r\to 0.
\]
Therefore, we can choose a modulus of continuity $\rho_1$ such
that
$$\rho_1(r)\ge C\sup_{m\ge 1}\frac{\tilde\rho(2^{m+2}
  r)+\tilde\rho(2^{m+2}
  r)^{1+\alpha\gamma}}{2^{m\alpha}}\quad\mbox{for all }r\ge 0, $$
and we have
\[
|u(y,s)-u(x,t)|\le \rho_1(|x-y|\vee|t-s|).
\]

If $|y-x|\ge r/2$, then 
\[
 \begin{split}
   |u(x,t)-u(y,s)|&\le |u(x,t)-u(x_0,t)|+|u(x_0,t)-u(y,s)|\\
   &\le \tilde\rho(r)+\tilde\rho(|x_0-y|\vee |t-s|)\\
   &\le \tilde\rho(2(|x-y|\vee|t-s|))+\tilde\rho((|x-y|+r)\vee |t-s|)\\
   &\le \tilde\rho(2(|x-y|\vee|t-s|))+\tilde\rho(3(|x-y|\vee |t-s|))\\
   &\le 2\tilde\rho(3(|x-y|\vee |t-s|)).
 \end{split}
\]

If $|x-y|\le r/2$ and
$|s-t|\ge r^{2+\gamma} (\tilde\rho(2r))^{-\gamma}/4$, then
$r\le
4^{\frac{1}{2+\gamma}}(2M+2)^\frac{\gamma}{2+\gamma}|s-t|^{\frac{1}{2+\gamma}}$
when $\gamma\ge 0$, and $r\le 2|s-t|^{\frac{1}{2}}$ when
$\gamma\le 0$. Then one can show similar to the above that
\[
 \begin{split}
  |u(x,t)-u(y,s)|&\le 2 \tilde\rho(c(|x-y|\vee |t-s|^{\frac 12}\vee |s-t|^{\frac{1}{2+\gamma}})),\\
  &\le \rho_2(|x-y|\vee |t-s|)
   \end{split}
\]
where $\rho_2(r)=2\tilde\rho(cr^{\frac 12})$ or
$\rho_2(r)=2\tilde\rho(cr^{\frac{1}{2+\gamma}})$ depending on whether
$\gamma\ge 0$ or $\gamma\le 0$ is a modulus of continuity, $c$ is a
positive constant depending only on $M$ and $\gamma$.

This finishes the proof in this first case.
\medskip

\emph{Case 2:} $r^{2+\gamma} (\tilde\rho(2r))^{-\gamma}\ge 1+t$.

\medskip

Then let $\lambda=\sqrt{|t+1|}$ when $\gamma\ge 0$, and
$\lambda=(2M+2)^{\frac{\gamma}{2+\gamma}}|t+1|^{\frac{1}{2+\gamma}}$
when $\gamma\in(-1,0)$. Then one can check that $\lambda\le r$.

If $|y-x|\le \lambda/2$ and
$|s-t|\le \lambda^{2+\gamma} (\tilde\rho(2\lambda))^{-\gamma}/4$, let
\[
v(z,\tau)=\frac{u(\lambda z+x,\lambda^{2+\gamma}
  (\tilde\rho(2\lambda))^{-\gamma}\tau+t)-u(x_0,t)}{\tilde\rho(2\lambda)}\quad\mbox{for
}(z,\tau)\in Q_1.
\]
Then 
\[
v_\tau=(|\nabla v |^2+\va^2r^2\tilde\rho(2\lambda)^{-2})^{\gamma/2}
\left(\delta_{ij}+(p-2)\frac{v_iv_j}{|\nabla
    u|^2+\va^2\lambda^2\tilde\rho(2\lambda)^{-2}}\right)u_{ij}
\quad\mbox{in}Q_1.
\]
Notice that
$\lambda^{2+\gamma} (\tilde\rho(2\lambda))^{-\gamma}\le
\lambda^2\le\lambda$
when $\gamma\ge 0$, and
$\lambda^{2+\gamma} (\tilde\rho(2\lambda))^{-\gamma}\le \lambda
r^{1+\gamma}(\tilde\rho(2r))^{-\gamma}\le\lambda$
when $\gamma\in(-1, 0)$. Thus, $|v(z,\tau)|\le 1$ for
$(z,\tau)\in Q_1$. Also,
$\va \lambda /\tilde\rho(2\lambda)\le \va \lambda
/\tilde\rho(\lambda)\le \va<1$.
Then, by the similar arguments in case 1, we have
\[
|u(y,s)-u(x,t)|\le \rho_1(|x-y|\vee|t-s|).
\]

If $|y-x|\ge \lambda/2$, then $|t+1|\le c(|x-y|^2\vee|x-y|^{2+\gamma})\le c|x-y|$ for some $c>0$ depending only on $M$ and $\gamma$. Therefore,
\[
 \begin{split}
   |u(x,t)-u(y,s)|&\le |u(x,t)-u(x,-1)|+|u(x,-1)-u(y,s)|\\
   &\le \tilde\rho(|t+1|)+\tilde\rho(|x-y|\vee |1+s|)\\
   &\le \tilde\rho(c|x-y|)+\tilde\rho((|x-y|)\vee |1+t|)\\
   &\le \tilde\rho(c(|x-y|\vee|t-s|))+\tilde\rho(c|x-y|\vee|t-s|)\\
   &= 2\tilde\rho(c(|x-y|\vee|t-s|))\\
   &\le \rho_2(|x-y|\vee|t-s|).
 \end{split}
\]

If $|x-y|\le \lambda/2$ and
$|s-t|\ge \lambda^{2+\gamma} (\tilde\rho(2\lambda))^{-\gamma}/4$, then
$\lambda\le
4^{\frac{1}{2+\gamma}}(2M+2)^\frac{\gamma}{2+\gamma}|s-t|^{\frac{1}{2+\gamma}}$
when $\gamma\ge 0$, and $\lambda\le 2|s-t|^{\frac{1}{2}}$ when
$\gamma\le 0$. Then one can show similar to the above that
\[
 \begin{split}
   |u(x,t)-u(y,s)|&\le |u(x,t)-u(x,-1)|+|u(x,-1)-u(y,s)|\\
   &\le \tilde\rho(|t+1|)+\tilde\rho(|x-y|\vee |1+s|)\\
   &\le \tilde\rho(c(|s-t|^{\frac{2}{2+\gamma}}\vee |s-t|^{\frac{2+\gamma}{2}}))+\tilde\rho((|x-y|)\vee |1+t|)\\
   &\le \tilde\rho(c(|s-t|^{\frac{1}{2+\gamma}}\vee |s-t|^{\frac{1}{2}}))+\tilde\rho(c(|s-t|^{\frac{1}{2+\gamma}}\vee |s-t|^{\frac{1}{2}}))\\
   &\le \rho_2(|x-y|\vee |t-s|).
 \end{split}
\]

This finishes the proof in this second case.

In the end, we deal with the situation that $\gamma\in(-1,0)$ and
$r^{1+\gamma}(2M+2)^{-\gamma}\ge 1$. Then $r\ge c$ for
$c=(2M+2)^{\frac{\gamma}{1+\gamma}}$. Let
$\lambda=(2M+2)^{\frac{\gamma}{2+\gamma}}|t+1|^{\frac{1}{2+\gamma}}$. There
exists $\mu>0$ depending only $M$ and $\gamma$ that if $|t+1|\le \mu$,
then $\lambda\le c$, $c^{2+\gamma} (\tilde\rho(2c))^{-\gamma}\ge 1+t$,
and $\lambda^{1+\gamma} (2M+2)^{-\gamma}\le 1$. Then, for
$t\le -1+\mu$, the same arguments in case 2 works without any change.

Now the final left case is that
$(x,t)\in \overline B_{1-c}\times[-1+\mu,0]$. Then we only need to
consider that $(y,s)\in B_{1-c/2}\times[-1+\mu/2,0]$. It follows from
Corollary \ref{cor:lip} and Lemma \ref{lem:holder in t} that there
exists a modulus of continuity $\bar\rho$ depending only on
$n,\gamma,p,M$ that
\[
u(x,t)-u(y,s)|\le \overline\rho(|x-y|\vee |t-s|).
\]
This finishes the final situation.

Then $\rho^*(r):=\rho_1(r)+\rho_2(r)+\bar\rho(r)$ is a desired modulus
of continuity. The proof of this theorem is thereby completed.
\end{proof}
%--------------

\bibliographystyle{abbrv}
\bibliography{reference}

\bigskip
\bigskip

\noindent C. Imbert

\noindent Department of Mathematics and Applications, CNRS \& \'Ecole Normale Sup\'erieure (Paris) \\ 45 rue d'Ulm, 75005 Paris, France\\[1mm]
Email: \textsf{Cyril.Imbert@ens.fr}

\bigskip

\noindent T. Jin

\noindent Department of Mathematics, The Hong Kong University of Science and Technology\\
Clear Water Bay, Kowloon, Hong Kong\\[1mm]
Email: \textsf{tianlingjin@ust.hk}  

\bigskip

\noindent L. Silvestre

\noindent Department of Mathematics, The University of Chicago\\
5734 S. University Avenue, Chicago, IL 60637, USA\\[1mm]
Email: \textsf{luis@math.uchicago.edu}

\end{document}